\documentclass{article}

\usepackage[T1]{fontenc}
\usepackage[a4paper, total={6.5in, 9.5in}]{geometry}
\usepackage{graphicx} 
\usepackage{appendix}
\usepackage{placeins}
\usepackage[show]{ed}
\usepackage{algorithm2e}
\usepackage[noend]{algorithmic}
\usepackage{amsmath,amsfonts,amsthm,amssymb,mathtools,bbm,bm,xcolor,hyperref,cleveref,accents,enumitem}
\usepackage[nocompress]{cite}
\definecolor{green}{rgb}{0.0, 0.7, 0.1}
\definecolor{red}{rgb}{0.85, 0.0, 0.0}
\definecolor{blue}{rgb}{0, 0, 0.8}
\hypersetup{
  colorlinks,
  citecolor={blue},
  linkcolor=red,
  urlcolor=purple}

\usepackage[margin=0.9375cm,format=hang, font={small,it} , labelfont=normalfont,bf]{caption}

\DeclareMathOperator*{\Exp}{Exp}
\DeclareMathOperator*{\tr}{tr}

\newcommand{\E}{\mathbb{E}}
\newcommand{\R}{\mathbb{R}}
\newcommand{\N}{\mathbb{N}}
\newcommand{\A}{\mathcal{A}}
\newcommand{\D}{\mathcal{D}}
\newcommand{\PP}{\mathbb{P}}

\newcommand{\Pcal}{\mathcal{P}}

\newcommand{\F}{\mathcal{F}}
\newcommand{\Lcal}{\mathcal{L}}

\newcommand{\W}{\mathcal{W}}
\newcommand{\pr}[1]{{#1}^{\prime}}

\newcommand{\de}{\mathrm{d}}
\newcommand{\Ncal}{\mathcal{N}}
\newcommand{\act}{\sigma_{\mathrm{act}}}

\numberwithin{equation}{section}
\numberwithin{figure}{section}
\numberwithin{table}{section}
\theoremstyle{plain}
\newtheorem{thm}{Theorem}[section]
\crefname{thm}{theorem}{theorems}
\theoremstyle{definition}

\theoremstyle{definition}

\theoremstyle{plain}
\newtheorem{prop}[thm]{Proposition}
\theoremstyle{plain}
\newtheorem{cor}[thm]{Corollary}
\theoremstyle{plain}
\newtheorem{lemma}[thm]{Lemma}
\theoremstyle{plain}
\newtheorem{ass}[thm]{Assumption}
\crefname{ass}{assumption}{assumptions}
\theoremstyle{plain}
\newtheorem{remark}[thm]{Remark}
\theoremstyle{definition}

\theoremstyle{definition}

\RestyleAlgo{ruled}
\makeatletter
\renewcommand{\@algocf@capt@plain}{above}
\makeatother

\crefname{algocf}{Algorithm}{Algorithms}
\Crefname{ALC@unique}{Line}{Lines}

\title{Extended mean field control: a global numerical solution via finite-dimensional approximation}
\author{Athena Picarelli\thanks{Department of Economics, University of Verona,\\
      \indent \indent \texttt{\{athena.picarelli, marco.scaratti\}@univr.it}.}
\and Marco Scaratti \footnotemark[1]
\and Jonathan Tam \thanks{Mathematical Institute, University of Oxford, \texttt{tam@maths.ox.ac.uk}.}}
\date{March 18, 2026}

\begin{document}

\maketitle
\begin{abstract}
We investigate the global numerical approximation of a class of extended mean field control problems (MFC), where the dynamics and costs depend on the joint distribution of the state and the control. We propose a framework to approximate the value function globally over the Wasserstein space, moving beyond the restriction of fixed initial conditions. Our approach exploits the propagation of chaos by approximating the infinite-dimensional MFC problem by an $N$-player cooperative game, together with the usage of finite-dimensional solvers. This method avoids the need to parametrise functions on an infinite-dimensional space, offering a balance between probabilistic rigor and computational efficiency.
\end{abstract}



\section{Introduction}
In this article, we consider a finite-dimensional approximation for a class of extended mean field control problems. Our method provides a global approximation of the value function on the infinite-dimensional Wasserstein space, as opposed to 
a fixed initial condition. 
We appeal to the propagation of chaos \cite{lacker2017limit, djete2022extended} and consider an approximating finite-player cooperative game, which allows us to fit a finite-dimensional function by proxy.\\

The study of mean field control (MFC) problems, arising from the control of SDEs of McKean-Vlasov type (MKV SDE), is closely related to the theory of mean field games (MFG), both of which act as an asymptotic approximation of stochastic games between a large number of players. In particular, MFC problems correspond to symmetric cooperative games, where the role of the limiting SDE is that of a central planner seeking a social optimum. MFC and MFG problems appear in many economics and financial applications, including planning in energy markets \cite{aid2020entry}, as well as optimal liquidation and price impact problems with aggregate flow \cite{huang2019mean,cardaliaguet2018mean,fu2024mean}. An extensive treatment of MFC and MFG problems can be found in the two-volume monograph \cite{carmona2018probabilistic}. From a control perspective, MFC problems can be viewed in the lens of a single-agent optimisation problem in which the dynamics and cost depend on the law of the controlled state process. Such formulations can arise in contexts such as optimal transport and target constraint problems, where related Schr\"odinger bridge problems have recently attracted attention in generative modelling in finance \cite{de2021diffusion,vargas2021solving,wang2021deep}.\\

Due to the nature of the MKV SDE, the state variable in the value function for an MFC problem is a probability measure in the Wasserstein space. Whilst the value function can be considered as the unique viscosity solution to a class of Master Bellman PDEs under certain conditions \cite{cosso2023optimal, cosso2023master}, it is typical in practice to numerically approximate the value function by particle methods. The finite-dimensional reduction of problems on infinite-dimensional spaces of probability measures have been considered in theoretical works \cite{gangbo_finite_2021, talbi_finite-dimensional_2024}, with particle approximation methods being considered in \cite{germain2022deepsets, dayanikli2024deep, carmonaa2023deep, soner2025learning}. It is typical to parametrise the control with a neural network, and minimising a loss function based on either a dynamic programming or BSDE characterisation of the control problem. This method has demonstrated empirical effectiveness at overcoming the curse of dimensionality associated with more traditional PDE-based methods \cite{han2017deep}.\\

The main contribution of this paper is to provide a numerical approximation of the MFC value function on the Wasserstein space globally. Given a value function $v$ for the MFC problem, the propagation of chaos yields an approximation of the form
    \begin{align}\label{intro_v}
        v(t,\mu) \approx \int_{(\R^d)^N} \bar{v}_N(t,x_1,\ldots, x_N) \de \mu^{\otimes N},
    \end{align}
    where $\bar{v}_N$ is the value function of an associated $N$-player cooperative game. We therefore approximate $\bar{v}_N$ globally, which is a finite-dimensional function. Afterwards, $v(t,\mu)$ can be computed by sampling $\bar{v}_N$ according to the measure $\mu$. This contrasts with the aforementioned works, where the MFC dynamics are simulated directly via particle methods, and solutions are only given for fixed initial conditions. To the best of our knowledge, \cite{pham2022mean, pham2023mean} provide a global approximation of the value function by another neural network parametrisation, followed by a fitting with regression. However, learning a function on the Wasserstein space comes at a large computational cost, and the estimation of the joint law $\Lcal(X_t, \alpha_t)$, where $X$ and $\alpha$ denote the state and control processes respectively, becomes an extra source of approximation error. By first passing to the finite-dimensional setting, and learning instead the associated (finite-dimensional) value function, we benefit from the following:\\

\begin{enumerate}
    \item \emph{More efficient training of the neural network.} The law of the state process $\Lcal(X_t)$ appear in MFC problems in either the coefficients of the state dynamics, the cost functions, or as an input in feedback controls of the form $a(t,X_t, \Lcal(X_t))$. A large number of trajectories (usually in the order of $10^6$) would be required for an accurate Monte Carlo estimate of the term $\Lcal(X_t)$.
    In the finite-dimensional problem, the mean field terms are replaced with the empirical distributions $\frac{1}{N} \sum^N_{n=1} \delta_{X^{n}_t}$. In this case, $N$ Brownian trajectories are sufficient for each estimate of the value function. This gives a higher flexibility on the batch size for training: a smaller batch size is useful for quicker gradient descent iterates, whilst a larger batch size can reduce the variance in the gradient updates for a more stable training process. \vspace{0.5em}
    \item \emph{Avoiding an \textit{a priori} discretisation of the Wasserstein space.} Training the finite-dimensional value function $\bar{v}_N$ does not require a discretisation of the underlying state space. Moreover, the evaluation of the integral in \eqref{intro_v} only requires evaluating sampled inputs through the trained neural network, so the Monte Carlo error can be reduced to the desired accuracy with modest memory capacity constraints. 
    Our method avoids accumulating errors associated with the quantisation or truncation of moments over multiple time steps that would occur when simulating the mean-field system.  \vspace{0.5em}
     \item \emph{Flexible choice of finite-dimensional solvers.} 
    Our method is independent of the choice of the finite-dimensional solver. Depending on the problem assumptions, dynamic programming or BSDE characterisations of the value function can be used as a basis for the algorithm \cite{han2017deep, germain2022deepsets, lefebvre2023differential}. In contrast, BSDE characterisations are much harder to obtain in the mean field regime when the coefficients depend on the law of the controls $\Lcal(\alpha_t)$, as in  \cite{acciaio2019extended}, which restricts the possible choice of algorithms. \vspace{0.5em}
    \item \emph{Approximating non-feedback controls.} In the spirit of the approach in \cite{cosso2023master}, we show that under mild assumptions, the MFC problem can always be approximated with a regularised problem through the addition of independent Brownian noise. The regularised problem then admits an optimal feedback control. We show in a numerical example in \Cref{sec:experiments} that our method is able to approximate a non-feedback control that depends only on its random initial condition. This justifies a neural network parametrisation of the control, which are implicitly feedback in nature.
\end{enumerate}~

The rest of the paper is organised as follows. In \Cref{sec:MFCproblem} we outline the probabilistic setting for the extended MFC problem and the properties associated with the value function. In \Cref{sec:finitedimapprox} we give the finite-dimensional approximation for the value function of the mean field problem, which forms the basis for our numerical scheme. In \Cref{sec:numerics} we present the numerical scheme and the computational steps in detail. Finally we conclude in \Cref{sec:experiments} with three numerical experiments.

\section{Extended mean field control problem}\label{sec:MFCproblem}
\subsection{Probabilistic framework}
Let $(\Omega, \mathcal{F}, \PP)$ be a complete probability space supporting an $m$-dimensional Brownian motion $W = (W_t)_{t\geq0}$. Let $\mathbb{F}^W = (\mathcal{F}^W_t)_{t\geq0}$ denote the $\PP$-completion of the natural filtration generated by $W$, which satisfies the usual conditions. 
We assume there exists a sub-$\sigma$-algebra of $\mathcal{G} \subseteq \mathcal{F}$ such that $\mathcal{G}$ and $\mathcal{F}^W_{\infty}$ are independent, and that there exists an $\mathcal{G}$-measurable random variable $U_{\mathcal{G}}:\Omega \to \R$ with uniform distribution on $[0,1]$. The assumptions on $\mathcal{G}$
imply that for every $\mu\in\mathcal{P}_2(\mathbb{R}^d)$, there exists a random variable $\xi \in L^2(\Omega,\mathcal{G},\PP;\mathbb{R}^d)$ such that $\Lcal(\xi) = \mu$
(cf. \cite[Lemma 2.1]{cosso2023optimal}). Finally, denote by $\mathbb{F} = (\F_t)_{t\geq 0}$ the filtration with
$\F_t = \mathcal{G} \vee \mathcal{F}^W_t$, which satisfies the usual conditions.

\subsection{The control problem}
We consider a finite horizon optimal control problem. Given a non-empty Polish space $(A,d_A)$, let $\A$ denote the set of $\mathbb{F}$-progressively measurable processes $\alpha: [0,T]\times \Omega \to A$. We write ${\Pcal_2^{X,\A}:=\Pcal_2(\R^{d} \times A)}$.  Let $[b,\sigma] : [0,T] \times \R^{d} \times A \times \Pcal_2^{X,\A} \to [\R^{d},\R^{d \times m}]$. For an initial time $t\in[0,T]$, we consider the state equation governed by a stochastic differential equation of McKean-Vlasov type, with the additional dependence on the \textit{joint law of the state and the control}, given by:
\begin{align}\label{SDE} 
    \de X_s= b\big(s, X_s, \alpha_s, \Lcal(X_s,\alpha_s)\big)\ \de s + \sigma\big(s, X_s, \alpha_s, \Lcal(X_s,\alpha_s)\big)\ \de W_s,\ X_t = \xi,
\end{align}
for $s\in(t,T]$, where $\xi\in L^2(\Omega,\F_t,\PP;\mathbb{R}^d)$ and $\alpha\in\A$. To emphasise the dependence of $X$ on $(t,\xi,\alpha)$, we will write $X^{t,\xi,\alpha}$ when needed. 

\begin{remark}
    The space $\Pcal_2(\R^d)$, as well as the counterpart $L^2(\Omega,\F_t,\PP;\R^d)$ for the initial condition, can be generalized to $\Pcal_q(\R^d)$, resp. $L^q(\Omega,\F_t,\PP;\R^d)$, for any $q\geq1$ and the results of the current section still hold true. 
\end{remark}

We impose the following assumptions on the coefficients to guarantee the existence and uniqueness of solutions to \Cref{SDE}. For simplicity, we use $\lvert\cdot\rvert$ to denote both the Euclidean norm on $\R^d$ as well as the Frobenius norm of an $\R^{d\times m}$ matrix. 

\begin{ass}\label{ass:ex_uniq_SDE}~\\ \vspace{-1em}
    \begin{enumerate}[label=(\roman*)]
        \item $b,\sigma$ are bounded and continuous.
        \item For all $(t,(x,\pr{x}),(\nu,\pr{\nu}),a)\in[0,T]\times (\R^d)^2 \times \big(\Pcal_2^{X,\A}\big)^2 \times A$, there exists a constant $C_{\mathrm{Lip}}\geq0$ such that
        \begin{align*}
            \lvert [b,\sigma](t,x,a,\nu) - [b,\sigma](t, \pr{x},a, \pr{\nu}) \rvert
            \leq\ C_{\mathrm{Lip}} \big(\lvert x - \pr{x} \rvert + \W_2(\nu, \pr{\nu}) \big).
        \end{align*}
    \end{enumerate}
\end{ass}

\begin{prop}\label{prop:solmkvsde}
    Given \Cref{ass:ex_uniq_SDE}, for any $t \in [0,T]$, ${\alpha \in \A}$, and $\xi,\pr{\xi}\in L^2(\Omega,\F_t,\PP;\mathbb{R}^d)$, there exists a unique, indistinguishable solution $(X^{t,\xi,\alpha}_r)_{r\in[t,T]}$ to \eqref{SDE}, satisfying
    \begin{align*}\label{estimate_1:exuniq}
        \E\Bigg[\sup_{r\in[t,T]}{\lvert X_r^{t,\xi,\alpha} \rvert^2}\Bigg]\leq C_1\big(1+\E[|\xi|^2]\big),
    \end{align*}
    for a positive constant $C_1$ independent of $t,\xi,\alpha$. Also, for all $ s \in [t,T]$,
    \begin{equation}\label{estimate_2:exuniq}
        \E\Bigg[\sup_{r\in[s,T]}{\lvert X_r^{t,\xi,\alpha}-X_r^{s,\pr{\xi},\alpha} \rvert^2}\Bigg]\leq C_2\Big(\E[|\xi-\pr{\xi}|^2] + \big(1+\E[\lvert \xi \rvert^2]+\E[\lvert \pr{\xi} \rvert^2]\big)\lvert s-t \rvert\Big),
    \end{equation}
    for a positive constant $C_2$ independent of $t,s,\xi,\pr{\xi},\alpha$. Moreover, the flow property holds, that is, for every $t \leq r \leq s \leq T$,
    \begin{align*}
        X_r^{t,\xi,\alpha} = X_r^{s,X_s^{t,\xi,\alpha},\alpha}, \ \PP\text{-a.s.}
    \end{align*}
\end{prop}
\begin{proof}
The proof relies on standard arguments; see e.g. \cite[Lemma 2.1, Remark 2.3, Proposition 3.3]{cosso_pham_2019} and \cite[Proposition 2.10]{cosso2023optimal}.
\end{proof}

 Let $f : [0,T] \times \R^{d} \times A \times \Pcal_2^{X,\A} \to \R$ and $g :\R^{d} \times \Pcal_2(\R^{d}) \to \R$. To define the mean field control problem, for a given admissible control $\alpha\in\mathcal{A}$ we consider the functional
\begin{align*}
    J(t,\xi,\alpha) = \E \Bigg[ \int^T_t f\big(s, X_s, \alpha_s, \Lcal(X_s,\alpha_s)\big)\ \de s + g\big(X_T, \Lcal(X_T)\big) \Bigg].
\end{align*}
We assume the following conditions apply to the running and terminal payoff functionals.
\begin{ass}\label{ass:reward}~\\ \vspace{-1em}
    \begin{enumerate}[label=(\roman*)]
        \item $f,g$ are bounded and continuous.
        \item For all $(t,(x,\pr{x}),(\mu,\pr{\mu}),(\nu,\pr{\nu}),a)\in[0,T]\times (\R^d)^2 \times \big(\mathcal{P}_2(\R^d)\big)^2\times\big(\Pcal_2^{X,\A}\big)^2 \times A$, there exists a constant $\pr{C}_{\mathrm{Lip}}\geq0$ such that
        \begin{align*}
            &\hspace{-2cm} \lvert f(t,x,a,\nu) - f(t, \pr{x},a, \pr{\nu}) \rvert + \lvert g(x,\mu)-g(\pr{x},\pr{\mu}) \rvert \\
            &\leq\ \pr{C}_{\mathrm{Lip}}\big(\lvert x - \pr{x} \rvert + \W_2(\nu, \pr{\nu}) + \W_2(\mu, \pr{\mu}) \big).
        \end{align*}
    \end{enumerate}
\end{ass}

\begin{remark}
    The boundedness conditions on $b,\sigma,f,g$ required in \Cref{ass:ex_uniq_SDE,ass:reward} can be relaxed with some standard growing conditions; see for example \cite[Assumptions $(\mathbf{A}_{A,b,\sigma})$, $(\mathbf{A}_{f,g})$, $(\mathbf{A}_{f,g})_{cont}$]{cosso2023optimal} or \cite[Assumption $(\mathbf{A1})$]{cosso_pham_2019}.
    In addition, the results of the current section hold true by considering a weaker Lipschitz assumption for the measure argument, that is, for all $(t,(x,\pr{x}),(\nu,\pr{\nu}),a)\in[0,T]\times (\R^d)^2 \times \big(\Pcal_2^{X,\A}\big)^2 \times A$, there exists a constant $\bar{C}_{\mathrm{Lip}}\geq0$ such that
    \begin{align*}
        \lvert [b,\sigma, f](t,x,a,\nu) - [b,\sigma, f](t, \pr{x},a, \pr{\nu}) \rvert
        \leq\ \bar{C}_{\mathrm{Lip}} \big(\lvert x - \pr{x} \rvert + \W_2(\nu_x, \pr{\nu}_x) \big),
    \end{align*}
    where $\nu_x \in \Pcal_2(\R^{d})$ denotes the marginal of $\nu$ on the state, i.e., $\nu_x(\cdot) = \nu(\cdot \times A)$.
    The intention is to retain these items to facilitate the exposition as a whole.
\end{remark}

Thanks to the assumptions above, the so-called \emph{lifted value function}, given by
\begin{align*}
    V(t,\xi) := \inf_{\alpha \in \A}J(t,\xi,\alpha)
\end{align*}
is well-defined. We now recall some important properties of the lifted value function, which will be employed in \Cref{sec:finitedimapprox}.

\begin{prop}\label{lemma:propertiesliftedvf}
    Under \Cref{ass:ex_uniq_SDE,ass:reward} the lifted value function $V$ satisfies the following.
    \begin{enumerate}
        \item $V$ is bounded.
        \item $V$ is jointly continuous, that is $V(t_n,\xi_n) \to V(t,\xi)$ as $n\to\infty$ for every couple $(t,\xi)\in[0,T]\times L^2(\Omega,\F_t,\PP;\R^d)$ and sequence $\{(t_n,\xi_n)\}_{n\in\N_+}$ where $t_n\in[0,T]$ and $\xi_n\in L^2(\Omega,\F_{t_n},\PP;\R^d)$ for all $n\in\mathbb{N}_+$, such that $|t_n-t|+\E[|\xi_n-\xi|^2]\to0$ as $n\to\infty$.
        \item $V$ is Lipschitz in the second argument, i.e., there exists $L\geq0$, independent of $t,\xi,\pr{\xi}$, such that
        \begin{align*}
             |V(t,\xi)-V(t,\pr{\xi})|\leq L\E[|\xi-\pr{\xi}|^2]^{1/2}, \hspace{10pt}\forall t\in[0,T], \xi,\pr{\xi}\in L^2(\Omega,\F_t,\PP;\R^d).
        \end{align*}
    \end{enumerate}
\end{prop}
\begin{proof}
    The boundedness of $V$ follows from the boundedness assumption of the coefficients $f$ and $g$.
    To prove the second claim, we first note that the Lipschitz continuity of $f,g$ implies the following joint continuity property: for any sequence $\{(x_m,\zeta_m)\}_{m\in\N}\subset\R^d\times\mathcal{P}_2(\R^d)$ converging to $(x,\zeta)\in\R^d\times\mathcal{P}_2(\R^d)$,  as $m\to\infty$, we have  \begin{align}\label{eq:jointcont}
\sup_{\substack{t\in [0,T], a\in A\\ \eta\in\mathcal{P}_2(A)}} \lvert f(t,x_m,a,\nu_m)-f(t,x,a,\nu) \rvert + \lvert g(x_m,\zeta_m)-g(x,\zeta) \rvert \xrightarrow[m\to\infty]{} 0,
    \end{align}
    where $\nu_m,\nu\in\mathcal{P}_2^{X,\A}$ have marginals $(\zeta_m,\eta)$ and $(\zeta,\eta)$, respectively.
    Let $\alpha\in\A$, and let $K\geq 0$ such that $\| f\|_\infty \leq K$. Assume without loss of generality $t_n\leq t$ (an analogous argument holds for $t_n>t$). Then, we have
    \begin{align*}
        &\ \lvert J(t_n,\xi_n,\alpha) - J(t,\xi,\alpha) \rvert \\
        =&\  \bigg| \E\bigg[\int_{t_n}^T f(r,X_r^{t_n,\xi_n,\alpha},\alpha_r,\Lcal(X_r^{t_n,\xi_n,\alpha},\alpha_r)) \de r - \int_t^T  f(r,X_r^{t,\xi,\alpha},\alpha_r,\Lcal(X_r^{t,\xi,\alpha},\alpha_r)) \de r\\
        &\phantom{{}={}} + g(X_T^{t_n,\xi_n,\alpha},\Lcal(X_T^{t_n,\xi_n,\alpha})) - g(X_T^{t,\xi,\alpha},\Lcal(X_T^{t,\xi,\alpha})) \bigg] \bigg|\\
        \leq&\ \E\bigg[\int_{t_n}^t \lvert f(r,X_r^{t_n,\xi_n,\alpha},\alpha_r,\Lcal(X_r^{t_n,\xi_n,\alpha},\alpha_r)) \rvert\ \de r\\
        &\phantom{{}={}} + \int_t^T \lvert f(r,X_r^{t_n,\xi_n,\alpha},\alpha_r,\Lcal(X_r^{t_n,\xi_n,\alpha},\alpha_r)) - f(r,X_r^{t,\xi,\alpha},\alpha_r,\Lcal(X_r^{t,\xi,\alpha},\alpha_r)) \rvert\ \de r\\
        &\phantom{{}={}} + \lvert g(X_T^{t_n,\xi_n,\alpha},\Lcal(X_T^{t_n,\xi_n,\alpha})) - g(X_T^{t,\xi,\alpha},\Lcal(X_T^{t,\xi,\alpha})) \rvert \bigg]\\
        \leq&\ K\lvert t-t_n \rvert + \E\Big[ \lvert g(X_T^{t_n,\xi_n,\alpha},\Lcal(X_T^{t_n,\xi_n,\alpha})) - g(X_T^{t,\xi,\alpha},\Lcal(X_T^{t,\xi,\alpha})) \rvert \Big]\\
        &\phantom{{}={}} + \E\bigg[\int_t^T\sup_{(r,a,\eta) \in [t,T] \times A \times \mathcal{P}_2(A)} \lvert f^{t_n, \xi_n}_{r}(a, \eta)  - f^{t, \xi}_{r}(a, \eta)  \rvert\ \de r \bigg],
    \end{align*}
    where $f^{t, \xi}_{r}(a, \eta)\coloneqq f(r, X^{t,\xi,\alpha}_r, a, \Lcal(X^{t,\xi,\alpha}_r, \eta)) $. 
     Observe that by taking the product measure $\otimes$ as a particular coupling, one has
    \begin{align*} \sup_{r\in[t,T]}\mathcal{W}_2^2\Big(\Lcal(X_r^{t_n,\xi_n,\alpha}),\Lcal(X_r^{t,\xi,\alpha})\Big) \leq \E\Bigg[ \sup_{r\in[t,T]}\lvert X_r^{t_n,\xi_n,\alpha} - X_r^{t,\xi,\alpha} \rvert^2 \Bigg],
    \end{align*}
 and by applying the estimate \eqref{estimate_2:exuniq} from \Cref{prop:solmkvsde}, it follows by the joint continuity property of $f$ and $g$ that 
    \begin{align*}
    \lvert J(t_n,\xi_n,\alpha) - J(t,\xi,\alpha) \rvert \xrightarrow[n\to\infty]{} 0,\hspace{10pt} \forall \alpha\in\A.
    \end{align*}
    Finally, taking the supremum over the control gives
    \begin{align*}
        \lvert V(t_n,\xi_n) - V(t,\xi) \rvert \leq \sup_{\alpha\in\A} \lvert J(t_n,\xi_n,\alpha) - J(t,\xi,\alpha) \rvert \xrightarrow[n\to\infty]{} 0.
    \end{align*}

    For the third claim, the proof relies on similar arguments using the estimate \eqref{estimate_2:exuniq} from \Cref{prop:solmkvsde}; see e.g. \cite[proof of Proposition 2.3]{cosso2023master}.
\end{proof}

\begin{remark}
    The joint continuity of $f,g$ stated in \eqref{eq:jointcont} is the minimal assumption to get the conclusions in \Cref{lemma:propertiesliftedvf}; note that the stronger Lipschitz continuity property we required in \Cref{ass:reward} will be needed for the finite-dimensional approximation in \Cref{sec:finitedimapprox}.
\end{remark}

Looking at the lifted value function, it turns out that it only depends on the law of $\xi$; such a property is referred to as the \emph{law invariance property}.
\begin{thm}\label{thm:LIP}
    Given \Cref{ass:ex_uniq_SDE,ass:reward}, for any $t \in [0,T]$ and $\xi, \pr{\xi} \in L^2(\Omega,\F_t,\PP;\R^{d})$ such that $\Lcal(\xi) = \Lcal(\pr{\xi})$, it holds
    \begin{align*}
        V(t, \xi) = V(t, \pr{\xi}).
    \end{align*}
    Therefore, one can define the \emph{intrinsic value function} $v: [0,T] \times \Pcal_2(\R^{d}) \to \R$ by
\begin{align}\label{intrinsic_val}
    v(t,\mu) \coloneqq V(t, \xi),\quad \forall (t,\mu) \in [0,T] \times \Pcal_2(\R^{d}),
\end{align}
where $\xi \in L^2(\Omega,\F_t,\PP;\R^{d})$, such that $\Lcal(\xi)=\mu$.
\end{thm}
\begin{proof}

    The proof of the result adapts the arguments presented in \cite[Theorem 3.6]{cosso2023optimal} to our setting. We sketch here the main steps and differences, referring to \cite[Section 3.3 and Appendix B]{cosso2023optimal} for technical details.\\

    {\textbf{Step 1. }}{\textit{Canonical representation. }} Suppose that there exist two $\mathcal{F}_t$-measurable random variables $U_\xi, U_{\pr{\xi}}$ with uniform distribution on $[0,1]$, being independent on $\xi,\pr{\xi}$ respectively. From \cite[Lemma B.2]{cosso2023optimal}, for any $\alpha\in\mathcal{A}$ there exists a $Prog(\mathbb{F}^{W,t})\otimes\mathcal{B}(\R^d)\otimes\mathcal{B}([0,1])$-measurable function, denoted by $\mathfrak{a}$, such that
    \begin{align*}
          &\mathfrak{a} : [0,T]\times\Omega\times\R^d\times[0,1] \to A, \quad (s, \omega, x, y) \mapsto \mathfrak{a}_s(x,y)(\omega),\\
         &\big( \xi, (\alpha_s)_{s\in[t,T]}, (W_s-W_t)_{s\in[t,T]} \big) \overset{\mathcal{L}}{=} \big( \xi, \mathfrak{a}_s(\xi,U_\xi)_{s\in[t,T]}, (W_s-W_t)_{s\in[t,T]} \big),
    \end{align*}
    where the notation $\mathfrak{a}_s(\xi,U_\xi)$ refers to the function $(s,\omega) \mapsto \mathfrak{a}_s(\xi(\omega),U_\xi(\omega))(\omega)$,
    $\overset{\mathcal{L}}{=}$ stands for equality in law on the space $(\Omega,\mathcal{F},\PP)$, and $\mathbb{F}^{W,t}=(\F^{W,t}_s)_{s\geq t}$ is the $\PP$-completition of the filtration generated by Brownian increments $(W_s-W_t)_{s\geq t}$ from time $t$ on, so that $Prog(\mathbb{F}^{W,t})$ denotes the progressive $\sigma$-field on $\Omega\times[0,T]$, relative to the filtration $\mathbb{F}^{W,t}$. Similarly, we have
    \begin{equation*}
        \big( \xi, (\alpha_s)_{s\in[t,T]}, (W_s-W_t)_{s\in[t,T]} \big) \overset{\mathcal{L}}{=} \big( \pr{\xi}, (\mathfrak{a}_s(\pr{\xi},U_{\pr{\xi}}))_{s\in[t,T]}, (W_s-W_t)_{s\in[t,T]} \big),
    \end{equation*}
    and also
    \begin{align*}
        \big( (X^{t,\xi,\alpha}_s)_{s\in[t,T]}, (\alpha_s)_{s\in[t,T]} \big) \overset{\mathcal{L}}{=} \big( (X^{t,\pr{\xi},\mathfrak{a}(\pr{\xi},U_{\pr{\xi}})}_s)_{s\in[t,T]}, (\mathfrak{a}_s(\pr{\xi},U_{\pr{\xi}}))_{s\in[t,T]} \big).
    \end{align*}

    {\textbf{Step 2. }}{\textit{Conclusions. }}  As a consequence of Step 1, we get 
    $J(t,\xi,\alpha) = J(t, \pr{\xi},\mathfrak{a}(\pr{\xi},U_{\pr{\xi}})) \geq V(t,\pr{\xi})$ for every admissible control $\alpha$. 
    By inverting the roles of $\xi$ and $\pr{\xi}$, we get the converse inequality, from which $V(t,\xi)=V(t,\pr{\xi})$ and the intrinsic value function $v$ is well-defined. It remains to show that $U_\xi, U_{\pr{\xi}}$ exist; the discrete case where $\mathcal{L}(\xi) = \sum_{i=1}^{m}p_i\delta_{x_i}$, for $\{x_1,\dots,x_m\}\subset\R^d$, $x_i\neq x_j$ for $i\neq j$ and $p_i>0$ with $\sum_{i=1}^{m}p_i = 1$ follows from \cite[Lemma B.3]{cosso2023optimal}. Given the continuity of $\xi \mapsto V(t,\xi)$ from \Cref{lemma:propertiesliftedvf}, the case for arbitrary random variables $\xi$ and  $\pr{\xi}$ follows from a monotone convergence argument. 
\end{proof}

The results from \Cref{lemma:propertiesliftedvf} and the law invariance property \Cref{thm:LIP}, ensure that the intrinsic value function $v$ holds the same properties for $V$:
\begin{cor}\label{cor:propertiesvf}
    Under \Cref{ass:ex_uniq_SDE,ass:reward} the value function $v$ satisfies the following.
    \begin{enumerate}
        \item $v$ is bounded.
        \item $v$ is jointly continuous (in the sense stated in \Cref{lemma:propertiesliftedvf}).
        \item $v$ is Lipschitz in the second argument, i.e. there exists $L>0$ such that
        \[
            |v(t,\mu)-v(t,\pr{\mu})|\leq L\mathcal{W}_2(\mu,\pr{\mu}), \hspace{10pt}\forall t\in[0,T], \mu,\pr{\mu}\in\mathcal{P}_2(\R^d).
        \]
    \end{enumerate}
\end{cor}

\section{Approximation of the value function}\label{sec:finitedimapprox}

 As the law $\mu$ of the initial state belongs to $\Pcal_2(\R^d)$, the value function $v$  is defined on an infinite-dimensional measure space. 
The main idea of the following section is to consider the extended mean field control problem as the limit, for $N\to\infty$, of a cooperative $N$-player stochastic differential game; then, in \Cref{sec:numerics}, the resulting approximation will be employed to construct a numerical scheme.  In particular, the approximation we propose exploits the link between the intrinsic value function $v$ and the value function of a finite-dimensional cooperative differential game, denoted in the following by $\bar v_N$ or $\bar v_{\varepsilon,N}$. This link has been theoretically established in  \cite[Appendix A]{cosso2023master}, in a more restrictive framework. If compared with  \cite{cosso2023master}, we allow ourselves to work under weaker assumptions on the coefficients, as we are not concerned with the formulation of the Master Bellman equation and the uniqueness of its viscosity solutions. Moreover, we extend the result in \cite{cosso2023master} by considering the additional dependence on the joint law of state and control in both drift and diffusion. \\

We present the approximation result considering separately the case of non-degenerate and degenerate diffusion $\sigma$. By \textit{non-degenerate} we mean that there exists a constant $\theta>0$ such that for all $(t,x,a,\nu)\in[0,T]\times\R^d\times A\times\mathcal{P}_2^{X,\A}$ one has
\begin{align*}
    \sigma\sigma^{\top}(t,x,a,\nu) \geq \theta \mathrm{I}_d,
\end{align*}
where $\mathrm{I}_d$ represents the $d$-dimensional identity matrix. In the case of a non-degenerate diffusion, we directly approximate by means of a cooperative game between a finite but large number of players.
This relies on a propagation of chaos argument which requires non-degeneracy in the diffusion term. When the diffusion is possibly degenerate, a regularisation argument can be employed to approximate the original problem with a non-degenerate one. Therefore, in the general case, our approximation can be summarised into two main steps:

\begin{enumerate}[label=\textbf{\Roman*.}]
    \item  (if needed) Approximation by a non-degenerate control problem.
    \item Approximation by means of a cooperative $N$-player stochastic differential game.
\end{enumerate}
 Regarding the first step, the approximation of the value function still remains defined on an infinite-dimensional space, but allows us in the second step to apply a propagation of chaos result; this motivates the general \Cref{ass:ex_uniq_SDE} to work with possibly degenerate diffusion $\sigma$. For non-degenerate diffusions, we only need the second approximation step.

\subsection{\texorpdfstring{Case 1: $\sigma$ non-degenerate}{Case 1: sigma non-degenerate}}\label{subsec:non-deg}
Given the mean field control problem, we introduce the corresponding finite-player game. Let $N\in\N_{+}$ and let $(\bar{\Omega},\bar{\F},\bar{\PP})$ be an (auxiliary) complete probability space, supporting independent $m$-dimensional Brownian motions $\bar{W}^1,\ldots,\bar{W}^N$. Consider the corresponding $\bar{\PP}$-completion of the natural filtration generated by $\bar{W}^1,\ldots,\bar{W}^N$, denoted by $\bar{\mathbb{F}}^{W}=(\bar{\F}^{W}_t)_{t\geq0}$. In analogy with the mean field control problem, we assume the existence of a sub-$\sigma$-algebra $\bar{\mathcal{G}}\subset\bar{\F}$ such that $\bar{\mathcal{G}}$ and $\bar{\F}^{W}_{\infty}$ are independent and $\mathcal{P}_2(\R^{Nd})=\{\Lcal(\bar{\xi}) \, | \, \bar{\xi}\in L^2(\bar{\Omega},\bar{\mathcal{G}},\bar{\PP};\R^{Nd})\}$; finally, let $\bar{\mathbb{F}}=(\bar{\F}_t)_{t\geq0}$ where $\bar{\F}_t = \bar{\F}^{W}_t \vee \bar{\mathcal{G}}$.\\

To consider the control problem for each player, we introduce $\bar{\mathcal{A}}^N$, the family of $\bar{\mathbb{F}}$-progressively measurable processes $\bar{\alpha}:[0,T]\times\bar{\Omega} \to A^N$, where $\bar{\alpha}=(\bar{\alpha}^1,\ldots,\bar{\alpha}^N)$. For every initial time $t\in[0,T]$, control $\bar{\alpha}\in\bar{\mathcal{A}}^N$ and initial conditions $\bar{\xi}=(\bar{\xi}^1,\ldots,\bar{\xi}^N)\in L^2(\bar{\Omega},\bar{\F}_t,\bar{\PP};\R^{Nd})$, we consider the system of fully coupled controlled stochastic differential equations
\begin{align}\label{eq:nplayer_dynamics}
    \de \bar{X}^n_s= b^n_N(s, \bar{X}_s, \bar{\alpha}_s)\ \de t + \sigma^n_N(s, \bar{X}_s, \bar{\alpha}_s)\ \de \bar{W}^n_s, \hspace{10pt} \bar{X}^n_t = \bar{\xi}^n,
\end{align}
for  $s\in[t,T]$ and $n=1,\ldots,N$, where $\bar{X}=\bar{X}^{t,\bar{\xi},\bar{\alpha}}=(\bar{X}^{1,t,\bar{\xi},\bar{\alpha}},\ldots,\bar{X}^{N,t,\bar{\xi},\bar{\alpha}})$, and  $b^n_N$,$\sigma^n_N$, $f^n_N$ and $g^n_N$ are defined, for $\mathbf{x}=(x_1,\ldots,x_N)$ and $\mathbf{a}=(a_1,\ldots,a_N)$:
\begin{align*}
    &[b^{n}_{N},\sigma^{n}_{N}, f^n_N] : [0,T]\times {\mathbb{R}^{Nd}}\times A^N \to [\mathbb{R}^d,\R^{d \times m}, \R],\quad &&g^{n}_{N} : {\mathbb{R}^{Nd}}  \to \mathbb{R} \\ &(t,\mathbf{x},\mathbf{a}) \mapsto [b,\sigma, f]\bigg(t,x_n,a_n,\frac{1}{N}\sum_{j=1}^{N}{\delta_{(x_j,a_j)}}\bigg),\quad &&\mathrm{\mathbf{x}} \mapsto g\bigg(x_n,\frac{1}{N}\sum_{j=1}^{N}{\delta_{x_j}}\bigg)    
    \end{align*}
for every $t\in[0,T]$ and $\bar{\mu} \in \Pcal_2(\R^{Nd})$, with $\bar{\xi}=(\bar{\xi}^1,...,\bar{\xi}^N)\sim\bar{\mu}=\mu^1\otimes\ldots\otimes\mu^N$. Then, the map
\begin{align}\label{vn_eq}
     v_{N}(t,\bar{\mu}) = \inf_{\bar{\alpha}\in\bar{\A}^N}\frac{1}{N}\sum_{n=1}^{N}\bar{\E}\Bigg[ \int_t^T f^n_N(s,\bar{X}_s,\bar{\alpha}_s)\ \de s + g^n_N(\bar{X}_T) \Bigg]
\end{align}
is well-defined (cf. \cite{djete2022extended}). 

\begin{thm}\label{thm:prop_chaos}
    Under \Cref{ass:ex_uniq_SDE,ass:reward}, and assuming $A$ is compact, let $t\in[0,T]$, $q>2$ and $\{\mu^n\}   _{n\in\N_{+}}\subset\mathcal{P}_q(\R^d)$ such that $\sup_{n\geq1}\frac{1}{N}\sum_{n=1}^{N}\int_{\R^d}\lvert x \rvert^q\mu^n(\de x)<\infty$. Then
    \begin{align*}
        \lim_{N\to\infty}\Bigg| 
v_{N}(t,\bar{\mu})-v\bigg(t,\frac{1}{N}\sum_{n=1}^{N}\mu^n\bigg) \Bigg|=0.
    \end{align*}
\end{thm}
\begin{proof}
    See \cite[Theorem 3.3]{djete2022extended}.
\end{proof}
\begin{cor}\label{cor:limit}
    Under \Cref{ass:ex_uniq_SDE,ass:reward} and $A$ compact, let $q>2$, $(t, \mu) \in [0,T]\times\Pcal_q(\R^d)$, and let $\mu^{\otimes N}:=\mu\otimes\ldots\otimes \mu $. It holds that
    \begin{align*}
        \lim_{N\to\infty} v_{N}(t,\mu^{\otimes N}) = v(t,\mu).
    \end{align*}
\end{cor}

Note that the value function $v_{N}$ is still defined on an infinite-dimensional space but comes from a cooperative $N$-player game, where the coefficients of the problem take values from finite-dimensional spaces. In the following result, we characterise the value function $v_{N}$ by averaging the players' (deterministic) positions, thereby obtaining the desired finite-dimensional approximation.

\begin{thm}\label{thm:finite_approx}
    Under \Cref{ass:ex_uniq_SDE,ass:reward}, let $(t,\mu)\in[0,T]\times\mathcal{P}_2(\R^d)$ and 
    \begin{align}\label{vbarn:eq}
        \bar{v}_{N}:[0,T]\times \R^{Nd} \to \R, \quad \bar{v}_{N}(t,x_1,\ldots,x_N) := v_{N}(t, \delta_{x_1} \otimes \ldots \otimes \delta_{x_N}).
    \end{align}
    It holds that 
    \begin{align}
    \label{eq:int_repr}
        v_{N}(t,\mu^{\otimes N}) =  \int_{\R^{Nd}}\bar{v}_{N}(t,x_1,\ldots,x_N)\ \mu(\de x_1) \ldots \mu(\de x_N).
    \end{align}
\end{thm}
\begin{remark}
    \Cref{eq:int_repr} simply stands for
    \begin{align*}
        v_{N}(t,\mu^{\otimes N})=\int_{\R^{Nd}}v_{N}(t,\delta_{x_1}\otimes\ldots\otimes\delta_{x_N})\ \mu(\de x_1)\ldots\mu(\de x_N),
    \end{align*}
    and the choice to write $v_{N}(t,\delta_{x_1}\otimes\ldots\otimes\delta_{x_N})$ as $\bar{v}_{N}(t,x_1,\ldots,x_N)$ is made to emphasise the fact that $\bar{v}_{N}$ lives on a finite-dimensional space.  
\end{remark}
\begin{proof}[Proof of \Cref{thm:finite_approx}]
     We report here the main steps of the proof, referring to \cite[equality (A.10)]{cosso2023master} for additional technical details. We first observe that $\bar{v}_{N}$ corresponds to the value function of the cooperative $N$-player game presented in the second approximation step, with \emph{deterministic} initial state $\bar{x}=(x_1,\ldots,x_N)$ instead of the random vector $\bar{\xi}=(\bar{\xi}^1,\ldots,\bar{\xi}^N)$. In addition, the optimal control problem involves the coefficients $b^n_N,\sigma^n_N,f^n_N,g^n_N$, which have been defined directly from $b,\sigma,f,g$, hence they satisfy \Cref{ass:ex_uniq_SDE,ass:reward}. By employing similar computations as in \Cref{lemma:propertiesliftedvf} and observing that
    \begin{align*}
        \mathcal{W}_2^2\Bigg(\frac{1}{N}\sum_{n=1}^{N}\delta_{X^n},\frac{1}{N}\sum_{n=1}^{N}\delta_{Y^n}\Bigg) \leq \bar{\E}\Bigg[\frac{1}{N}\sum_{n=1}^{N}\lvert X^n-Y^n \rvert^2\Bigg],
    \end{align*}
    for $\{X^n\}_{n=1}^{N}, \{Y^n\}_{n=1}^{N} \subset L^2(\bar{\Omega},\bar{\F},\bar{\PP};\R^d)$, it follows that $\bar{v}_{N}$, and clearly ${v}_{N}$, satisfy the properties of \Cref{cor:propertiesvf}, suitably adapted to the context. Identity \eqref{eq:int_repr} follows as a special case of the more general equality
    \begin{align*}
        v_{N}(t,\bar{\mu}) = \int_{\R^{Nd}}\bar{v}_{N}(t,x^1,\ldots,x^N)\ \bar{\mu}(\de x^1,\ldots,\de x^N), \quad \forall (t,\bar{\mu})\in[0,T]\times\mathcal{P}_2(\R^{Nd}),
    \end{align*}
    from which \eqref{eq:int_repr} follows upon choosing $\bar{\mu}=\mu^{\otimes N}$; the latter is equivalent to proving
    \begin{align}\label{eq:int_repr_expect}
        v_{N}(t,\bar{\mu})=\bar{\E}[\bar{v}_{N}(t,\bar{\xi})], \quad \forall (t,\bar{\xi})\in[0,T]\times L^2(\bar{\Omega},\bar{\F}_t,\bar{\PP};\R^{Nd})\ \mathrm{s.t.}\ \Lcal(\bar{\xi})=\bar{\mu}.
    \end{align}

    The proof of \eqref{eq:int_repr_expect} consists of showing the double inequality $v_{N}(t,\bar{\mu}) \leq \bar{\E}[\bar{v}_{N}(t,\bar{\xi})]$ and $v_{N}(t,\bar{\mu}) \geq \bar{\E}[\bar{v}_{N}(t,\bar{\xi})]$ (cf. \cite[Theorem A.7]{cosso2023master}). This is possible first considering $\bar{\xi}\in L^2(\bar{\Omega},\bar{\F}_t,\bar{\PP};\R^{Nd})$ taking only a \emph{finite number} of values, that is $\bar{\xi}=\sum_{k=1}^{K}\bar{x}_k\mathrm{1}_{\bar{E}_k}$, for some $K\in\N_+, \bar{x}_k\in\R^{Nd}$ and $\bar{E}_k\subset\sigma(\bar{\xi})$ with $\{\bar{E}_k\}_{k=i}^{K}$ being a partition of $\bar{\Omega}$. The general case $\bar{\xi}\in L^2(\bar{\Omega},\bar{\F}_t,\bar{\PP};\R^{Nd})$ can be deduced from approximation arguments, recalling that $\bar{v}_{N}$ satisfies the properties of \Cref{cor:propertiesvf}, i.e. is bounded, jointly continuous and Lipschitz in $\bar{x}$.
\end{proof}

Finally, we state the main result.
\begin{thm}\label{thm:n_approx}
    Under \Cref{ass:ex_uniq_SDE,ass:reward}, and assuming $A$ is compact, it holds that, for every $(t,\mu)\in[0,T]\times\mathcal{P}_q(\R^d)$,    $q > 2$, 
    \begin{align}\label{eq:main_res}
     v(t,\mu) = \underset{N\to\infty}\lim{}\int_{\R^{Nd}}\bar{v}_{N}(t,x_1,\ldots,x_N)\ \mu(\de x_1) \ldots \mu(\de x_N) .
    \end{align}
\end{thm}
\begin{proof}
    The proof follows directly from \Cref{cor:limit} and \Cref{thm:finite_approx}.
\end{proof}

\begin{remark}\label{rem:convergencecontrolrules}
    We recall an additional important approximation result. We presented an approximation of the value function, through its finite-dimensional counterpart, which corresponds to the value function $\bar{v}_N$ of a cooperative $N$-player stochastic differential game. In \Cref{sec:numerics}, we approximate $\bar{v}_N$ by learning $\delta_N$-optimal controls of the large population cooperative control problem, for $\delta_N\searrow0$, which can be shown to converge towards an optimal control of the original mean field control problem. Note that the convergence of the control is provided in the ``weak sense'', given by measure-valued rules. We refer to the most recent result in literature \cite[Proposition 3.4]{djete2022extended}.
\end{remark}

\subsection{\texorpdfstring{Case 2: $\sigma$ degenerate}{Case 2: sigma degenerate}}\label{sigma_degen}

As remarked before, if $\sigma$ is degenerate, we first need an approximation with a non-degenerate problem in order to apply \Cref{thm:prop_chaos}. We present in detail the first approximation step and briefly sketch the second one, which is a straightforward adaptation of \Cref{subsec:non-deg}.

\subsubsection{Approximation by a non-degenerate control problem}
We consider an auxiliary complete probability space $(\hat{\Omega}, \hat{\mathcal{F}}, \hat{\PP})$ supporting a $m$-dimensional Brownian motion $\hat{W} = (\hat{W}_t)_{t\geq0}$ and a new additional $d$-dimensional Brownian motion $\hat{B} = (\hat{B}_t)_{t\geq0}$, independent of $\hat{W}$. Denote by $\hat{\mathcal{A}}$ the control process space and by $\hat{\mathbb{F}} = (\hat{\F}_t)_{t\geq 0}$ the standard filtration, both defined in analogy to \Cref{sec:MFCproblem}. Now, consider $\varepsilon>0$, $\hat{\xi}\in L^2(\hat{\Omega},\hat{\mathcal{F}}_t,\hat{\PP};\R^d)$,  $\hat{\alpha}\in\hat{\mathcal{A}}$ and denote by $\hat{X}^{\varepsilon,t,\hat{\xi},\hat{\alpha}}$ the unique solution of the non-degenerate McKean-Vlasov stochastic differential equation
\begin{align*}
    \de \hat{X}_s= b\big(s, \hat{X}_s, \hat{\alpha}_s, \Lcal(\hat{X}_s,\hat{\alpha}_s)\big)\ \de s + \sigma\big(s, \hat{X}_s, \hat{\alpha}_s, \Lcal(\hat{X}_s,\hat{\alpha}_s)\big)\ \de \hat{W}_s + \varepsilon\ \de \hat{B}_s,\ \hat{X}_t = \hat{\xi}.
\end{align*}
We then define the lifted valued function
\begin{align*}
    V_{\varepsilon}(t,\hat{\xi})\coloneqq\inf_{\hat{\alpha}\in\hat{\A}}J(t,\hat{\xi},\hat{\alpha}) = \inf_{\hat{\alpha}\in\hat{\A}}\hat{\E} \Bigg[ \int^T_t f\big(s, \hat{X}_s, \hat{\alpha}_s, \Lcal(\hat{X}_s,\hat{\alpha}_s)\big)\ \de s + g\big(\hat{X}_T, \Lcal(\hat{X}_T)\big) \Bigg].
\end{align*}
The law invariance property, as stated in \Cref{thm:LIP} holds, and we can define the intrinsic value function $ v_{\varepsilon}: [0,T] \times \Pcal_2(\R^{d})\to \R$. Clearly, the properties of \Cref{cor:propertiesvf} hold the same for $v_{\varepsilon}$. In addition, since $\hat{B}$ is independent of $\hat{W}$, we have that $v_0(t,\mu)=v(t,\mu)$, for all $(t,\mu) \in [0,T] \times \Pcal_2(\R^{d})$. We then formulate the first limit result.

\begin{lemma}\label{lemma:limit-v_eps}
    Under \Cref{ass:ex_uniq_SDE,ass:reward}, there exists a non-negative constant $\hat{C}$, depending only on $C_{\mathrm{Lip}}, \pr C_{\mathrm{Lip}}$ and $T$, such that, for every $\varepsilon\geq0$ we have
    \begin{align*}
        \lvert v_\varepsilon(t,\mu)-v_0(t,\mu) \rvert \leq \hat{C}\varepsilon,
    \end{align*}
    for all $(t,\mu)\in[0,T]\times\mathcal{P}_2(\R^d)$.
\end{lemma}
\begin{proof}
    By employing similar arguments as in \cite[Theorem 2.5.9]{krilov} and \cite[Lemma A.5]{cosso2023master}, one gets
    \begin{align*}
        \hat{\E}\Bigg[\sup_{s\in[t,T]}\lvert X_s^{\varepsilon,t,\hat{\xi},\hat{\alpha}}-X_s^{0,t,\hat{\xi},\hat{\alpha}} \rvert^2 \Bigg] \leq C\varepsilon^2,
    \end{align*}
    where $C\equiv C(C_{\mathrm{Lip}},T)\geq0$, for every $(\hat{\xi},\hat{\alpha})\in L^2(\hat{\Omega},\hat{\F}_t,\hat{\PP};\R^{d})\times\hat{\A}$. Finally, with standard techniques as the ones shown in \Cref{lemma:propertiesliftedvf}, we have
    \begin{align*}
        \lvert v_\varepsilon(t,\mu)-v_0(t,\mu) \rvert \leq 2\pr C_{\mathrm{Lip}}(T+1)\sqrt{C\varepsilon^2}=: \hat{C}\varepsilon.
    \end{align*}
\end{proof}

\subsubsection{Approximation by a cooperative stochastic differential game of finitely many players}
Let $N\in\N_+$ and let $(\tilde{\Omega},\tilde{\F},\tilde{\PP})$ an auxiliary complete probability space, supporting $\tilde{W}^1,\ldots,\tilde{W}^N$ independent $m$-dimensional Brownian motions, $\tilde{B}^1,\ldots,\tilde{B}^N$ independent $d$-dimensional Brownian motions. Finally define $\tilde{\mathbb{F}}=(\tilde{\F}_t)_{t\geq0}$ according to \Cref{subsec:non-deg}.\\

We introduce $\tilde{\mathcal{A}}^N$ the family of $\tilde{\mathbb{F}}$-progressively measurable processes $\tilde{\alpha}:[0,T]\times\tilde{\Omega} \to A^N$, where ${\tilde{\alpha}=(\tilde{\alpha}^1,\ldots,\tilde{\alpha}^N)}$. For every initial time $t\in[0,T]$, control $\tilde{\alpha}\in\tilde{\mathcal{A}}^N$ and initial conditions $\tilde{\xi}=(\tilde{\xi}^1,\ldots,\tilde{\xi}^N)\in L^2(\tilde{\Omega},\tilde{\F}_t,\tilde{\PP};\R^{Nd})$, we consider the system of fully coupled controlled stochastic differential equations
\begin{align*} 
    \de \tilde{X}^i_s= b^n_N(s, \tilde{X}_s, \tilde{\alpha}_s)\ \de t + \sigma^n_N(s, \tilde{X}_s, \tilde{\alpha}_s)\ \de \tilde{W}^n_s + \varepsilon\ \de \tilde{B}^n_s, \hspace{10pt} \tilde{X}^n_t = \tilde{\xi}^n,
\end{align*}
for  $s\in[t,T]$ and $n=1,\ldots,N$, where $\tilde{X}=\tilde{X}^{t,\tilde{\xi},\tilde{\alpha}}=(\tilde{X}^{1,t,\tilde{\xi},\tilde{\alpha}},\ldots,\tilde{X}^{N,t,\tilde{\xi},\tilde{\alpha}})$, and the coefficients $b^n_N,\sigma^n_N,f^n_N,g^n_N$ are defined as in \Cref{subsec:non-deg}. Again, for every $t\in[0,T]$ and $\tilde{\mu} \in \Pcal_2(\R^{Nd})$, with $\tilde{\xi}=(\tilde{\xi}^1,...,\tilde{\xi}^N)\sim\tilde{\mu}=\tilde{\mu}^1\otimes\ldots\otimes\tilde{\mu}^N$, it is well defined 
\begin{align*}
     v_{\varepsilon,N}(t,\tilde{\mu}) = \inf_{\tilde{\alpha}\in\tilde{\A}^N}\frac{1}{N}\sum_{n=1}^{N}\tilde{\E}\Bigg[ \int_t^T f^n_N(s,\tilde{X}_s,\tilde{\alpha}_s)\ \de s + g^n_N(\tilde{X}_T) \Bigg]
\end{align*}

According to \Cref{thm:prop_chaos} the following lemma holds.
\begin{lemma}\label{lemma:limit-v_eps_n}
    Under \Cref{ass:ex_uniq_SDE,ass:reward} and $A$ compact, let $(t, \mu) \in [0,T]\times\Pcal_q(\R^d)$, $q > 2$,  then it holds that
    \begin{align*}
        \lim_{n\to\infty} v_{\varepsilon,N}(t,\mu^{\otimes N}) = v_{\varepsilon}(t,\mu).
    \end{align*}
\end{lemma}
 Moreover, analogously to \Cref{thm:finite_approx}, one has the following representation for  $v_{\varepsilon,N}$.
\begin{thm}\label{thm:int-repr-v_eps_n}
    Under \Cref{ass:ex_uniq_SDE,ass:reward}, let $(t,\mu)\in[0,T]\times\mathcal{P}_2(\R^d)$ and consider the map
    \begin{align*}
        \bar{v}_{\varepsilon,N}:[0,T]\times\R^{Nd}, \quad \bar{v}_{\varepsilon,N}(t,x_1,\ldots,x_N) := v_{\varepsilon,N}(t, \delta_{x_1} \otimes \ldots \otimes \delta_{x_N}).
    \end{align*}
    Then it holds that 
    \begin{align*}
        v_{\varepsilon,N}(t,\mu^{\otimes N}) =  \int_{\R^{Nd}}\bar{v}_{\varepsilon,N}(t,x_1,\ldots,x_N)\ \mu(\de x_1) \ldots \mu(\de x_N).
    \end{align*}
\end{thm}

\subsubsection{Approximation result}
 By \Cref{lemma:limit-v_eps_n}, \ref{lemma:limit-v_eps}, and \Cref{thm:int-repr-v_eps_n}, we can finally state the main approximation result for diffusions that are possibly degenerate.
\begin{thm}\label{thm:finite_approx_eps}
    Under \Cref{ass:ex_uniq_SDE,ass:reward}, and assuming $A$ is compact, it holds that, for every $(t,\mu)\in[0,T]\times\mathcal{P}_{q}(\R^d)$,  $q > 2$,
    \begin{align*}
        v(t,\mu) =  \lim_{\varepsilon \to 0} \lim_{N \to \infty} {}\int_{\R^{Nd}}\bar{v}_{\varepsilon, N}(t,x_1,\ldots,x_N)\ \mu(\de x_1) \ldots \mu(\de x_N).
    \end{align*}
\end{thm}
\begin{proof}
    The proof follows from \Cref{lemma:limit-v_eps}, \Cref{lemma:limit-v_eps_n} and \Cref{thm:int-repr-v_eps_n}.
\end{proof}

\begin{remark}
    Similarly to \Cref{rem:convergencecontrolrules}, we stress here that $\bar{v}_{\varepsilon,N}$ is approximated by learning $\delta^\varepsilon_n$-optimal controls of the large population cooperative control problem, for $\delta^\varepsilon_n\searrow0$, converging in the sense of measure-valued rules towards an optimal control of the starting mean field control problem. This convergence is only provided for every fixed $\varepsilon>0$, and is not uniform in $\varepsilon$.
\end{remark}
 The result of \Cref{thm:finite_approx} (resp. \Cref{thm:finite_approx_eps}) implies that for a sufficiently large $N$ (and sufficiently small $\varepsilon$), a simple integration of the finite dimensional value function $\bar v_N$ (resp. $\bar v_{\varepsilon,N}$) can be used as an approximation of the value function of the extended mean-field optimal control problem. 
We provide an implementable algorithm to obtain this approximation in the next section.

\begin{remark}
    Under stronger assumptions, the value function $v(t,\mu)$ is a viscosity solution of the Master Bellman equation:
    \begin{align*}
    \begin{cases}
        \partial_t v(t,\mu) = F\big(t,\mu,v(t,\mu),\partial_{\mu}v(t,\mu)(\cdot),\partial_x\partial_{\mu}v(t,\mu)(\cdot)\big), &(t,\mu)\in[0,T)\times\mathcal{P}_2(\R^d),\\
        \displaystyle v(T,\mu) = \int_{\R^d}g(x,\mu)\ \mu(\de x), &\mu\in\mathcal{P}_2(\R^d),
    \end{cases}
    \end{align*}
    where $F$ is the associated maximised Hamiltonian for the functions $f,b,\sigma$, see e.g. \cite[Equation (3.3)]{cosso2023master} or \cite[Equation (5.1)]{cosso2023optimal}.
    In this case, the numerical method of \Cref{sec:numerics} can be employed to approximate the possibly unique solution of the Master Bellman equation.
\end{remark}

\section{A computational scheme with neural networks}\label{sec:numerics}

We detail here the precise steps for the numerical computation of the mean field value function \eqref{intrinsic_val}. As outlined in the introduction, we avoid learning directly the value function in the Wasserstein space, but instead we first learn an approximating finite-dimensional value function $\bar v_N$ (or $\bar v_{\varepsilon,N}$) before evaluating a corresponding integral for the input measure.  We point out that $\bar v_N$ (resp. $\bar v_{\varepsilon,N}$) is the value function associated to an classical optimal control problem in dimension $Nd$ and is the unique viscosity solution to a Hamilton-Jacobi-Bellman (HJB) equation under our assumptions.\\

For simplicity, we fix a time horizon of $[0,T]$ and large $N\in \N_+$ (the choice of $N$ depends on the specific problem, see \Cref{sec:experiments}).  Due to the potentially high dimensionality of the problem, classical numerical methods for HJB equations, such as finite differences, finite elements and semi-Lagrangian schemes,  suffer from the well known ``curse of dimensionality''.  For this reason, the algorithm we propose to numerically compute $\bar v_N$ (resp. $\bar v_{\varepsilon,N}$) can make use of neural networks so that it is implementable in high dimension. As we already mentioned in the introduction, the choice of the numerical method used to approximate the finite dimensional value function  is not restrictive and the one proposed below (Sections \ref{sect:policy} and \ref{sect:value}) is just a possible one having the advantage of working in the most general setting.
\\

For the purposes of numerical computation, we shall assume that the control space $A$ is a subset of some Euclidean space $\R^{d_A}$. For our experiments, we will parametrise both the control and value functions with neural networks. We adopt a similar network structure to that used in \cite{lefebvre2023differential}. Firstly, the output of two sub-networks, taking the time $t$ and the particles $(x_1,\ldots, x_N)$ as inputs respectively, are concatenated. Then, this concatenated output is passed as an input to one final sub-network. Each sub-network is of feedforward form, comprising of one hidden layer. 
 By a feedforward network we mean a map $\Ncal\Ncal: \R^{w_0} \to \R^{w_l}$, specified by the number of layers $l$, the widths $\bm{w} = (w_0, \ldots, w_l)$ of each layer and an activation function $\act: \R \to \R$. For an input $x \in \R^{w_0}$ we have
    \begin{align*}
        \Ncal\Ncal(x) = \sigma_{\rm out}(\mathrm{W}_l(\act(\mathrm{W}_{l-1}(\ldots \act(\mathrm{W}_1 x + b_1)) + b_{l-1} )) + b_l)
    \end{align*}
    where $\mathrm{W}_j\in\R^{w_{j}\times w_{j-1}}, b_j\in\R^{w_{j}}$ for $1\leq j \leq l$ are known as the weights and biases of the network, and the activation function $\act$ is applied component-wise to vectors. The function $\sigma_{\rm out}: \R \to \R$, applied component-wise on the last layer, serves to modify the range of the neural network. In contrast to \cite{lefebvre2023differential}, the width of the hidden layer to grows linearly with respect to the dimension of the players; in our experiments this is taken to be $10+ Nd$. This is done to allow for a sufficient number of parameters in the network to avoid underfitting.

\subsection{Obtaining an approximately optimal policy}\label{sect:policy}

The first step is to obtain an approximately optimal policy for the finite $N$-player game through a policy gradient iteration. To do so, we first approximate the system of $N$-players $(\bar{X}^1, \ldots, \bar{X}^N)$ in \eqref{eq:nplayer_dynamics} with the Euler-Maruyama scheme as follows. We discretise the time horizon $[0,T]$ into equal subintervals of length $\Delta t>0$, and let $\Delta W^n_s \coloneqq W^n_{s+\Delta t} - W^n_s$ denote the Brownian increments. The approximating time-discretised system is denoted by $(\hat{X}^1, \ldots, \hat{X}^N)$, which for each $n \in \{1, \ldots, N\}$ and $s\in \{0, \Delta t, 2 \Delta t,\ldots, T-\Delta t\}$, satisfies the dynamics
\begin{align}\label{traj_EM}
    \hat{X}^n_{s+\Delta t} = \hat{X}^n_s + b^n_N(s, \hat{X}_s, \alpha_s) \Delta t + \sigma^n_N(s, \hat{X}_s, \alpha_s) \Delta W^n_s,
\end{align}
where the functions $b^n_N$, $\sigma^n_N$, $f^n_N$ and $g^n_N$ are as given in \Cref{sec:finitedimapprox}. If $\sigma$ is possibly degenerate,  we simulate additional Brownian increments $\Delta B^n_t = B^n_{t+\Delta t}-B^n_t$, independent of $\Delta W_t$, and for some small $\varepsilon > 0$, consider instead
\begin{align*}
    \hat{X}^n_{s+\Delta t} = \hat{X}^n_s + b^n_N(s, \hat{X}_s, \alpha_s) \Delta t + \sigma^n_N(s, \hat{X}_s, \alpha_s) \Delta W^n_s + \varepsilon \Delta B^n_s.
\end{align*}
We now parameterise the control by a neural network with parameters $\theta$, and denote this control by $\alpha_{\theta}$. Choose a measure $\mu_0 \in \Pcal_2(\R^d)$ to act as a sampling measure for $\hat{X}_0$ (see \Cref{rem:sampling_measure} on the choice of such a sampling measure), and consider the (discretised) cost functional
\begin{align*}
     J_{N}(\theta) = \frac{1}{N} \sum^N_{n=1}\ \E_{\hat{X}_0 \sim \mu_0} \left[ \sum^{T-\Delta t}_{s=0} f^n_N(s, \hat{X}_s, \alpha_s) \Delta t + g^n_N(\hat{X}_T) \right].
\end{align*}
The gradient of the cost functional with respect to the parameters $\theta$ is given by
\begin{align*}
    \nabla_{\theta} J_N(\theta) = \frac{1}{N} \sum^N_{n=1}\ \E_{\hat{X}_0 \sim \mu_0} \left[ \sum^{T-\Delta t}_{s=0} \nabla_{\theta} f^n_N(s, \hat{X}_s, \alpha_{\theta}(s, \hat{X}_s)) \Delta t + \nabla_{\theta} g^n_N(\hat{X}_T) \right].
\end{align*}
We then estimate $\nabla_{\theta} J_N(\theta)$ with Monte Carlo simulation. Let $M>0$ denote the total number of sample trajectories of the system $\hat{X}$, given by \eqref{traj_EM}. For each $1 \leq m \leq M$ and $1 \leq n \leq N$, let $\hat{X}^{n,m}_s$ denote the $m$\textsuperscript{th} sample of the $n$\textsuperscript{th} player at time $s$. The Monte Carlo estimator of $\nabla_{\theta} J_N(\theta)$ is then given by
\begin{align*}
    G_{\theta} \coloneqq\frac{1}{MN} \sum^M_{m=1} \sum^N_{n=1}  \sum^{T - \Delta t}_{s=0} \nabla_{\theta} f^n_N(s, \hat{X}^{\cdot,m}_s, \alpha_{\theta}(s, \hat{X}^{\cdot,m}_s))\Delta t + \nabla_{\theta} g^n_N(\hat{X}^{\cdot,m}_T).
\end{align*}
To apply the policy iteration, first initialise the parameters of the neural network by some $\theta_0$, as well as a learning rate $\gamma > 0 $. Then the successive iterations for the parameters are given by the update rule
\begin{align*}
    \theta_{j+1} \coloneqq \theta_j - \gamma G_{\theta_j},\quad j \geq 0.
\end{align*}
At each iteration $j$, the cost functional $J_N(\theta_j)$ is evaluated. The iteration procedure terminates when no further minimisation in $J_{N}$ occurs after a prespecified number of iterations. We take the neural network after termination as an approximately optimal control, and denote this by $\alpha^{*}$. We note that although each $\hat{X}^{n,m}_t$ depend implicitly on the parameters $\theta$, the terms $\nabla_{\theta} g^n_N(\hat{X}^{\cdot,m}_T)$ and $\nabla_{\theta} f^n_N(s, \hat{X}^{\cdot,m}_s, \alpha_{\theta}(s, \hat{X}^{\cdot,m}_s))$ can be computed efficiently in practice with automatic differentiation tools.

\begin{remark}\label{rem:sampling_measure}
    The sampling measure $\mu_0$ should have support on the domain of $X_0$ and should be chosen such that there is a sufficient sampling of the domain. Alternatively, one can also truncate the domain and restrict the initial conditions within some compact set $\mathcal{K} \subset \R^d$, sampling $\mathcal{K}$ according to the uniform distribution.
\end{remark}

\subsection{Computing the finite-dimensional value function}\label{sect:value}

For the rest of this section, we will assume that the diffusion coefficient $\sigma$ is non-degenerate and apply the results from \Cref{subsec:non-deg}. For the case that $\sigma$ is possibly degenerate, it is sufficient to refer to \Cref{sigma_degen}, and substitute the value function $\bar{v}_N$ with $\bar{v}_{\varepsilon,N}$ (for $\varepsilon$ small).\\

Recall from \eqref{vn_eq} and \eqref{vbarn:eq} that $\bar{v}_N$ is the value function for the $N$-player game, and that for any $t\in [0,T]$, $(x_1,\ldots, x_N)\in (\R^{d})^N$, we have 
\begin{align*}
     \bar{v}_{N}(t, x_1, \ldots, x_N) =\inf_{\bar{\alpha}\in\bar{\A}^N}\frac{1}{N}\sum_{n=1}^{N}\bar{\E}\Bigg[ \int_t^T f^n_N(s,\bar{X}_s,\bar{\alpha}_s)\ \de s + g^n_N(\bar{X}_T) \Bigg].
\end{align*}
By using the discretised dynamics $\hat{X}$ and approximately optimal control $\alpha^{*}$ as obtained in the previous section, given an initial condition $\hat{X}^{\cdot, m}_t = (x_1, \ldots, x_N)$, an estiamte of $\bar{v}_{N}(t, x_1, \ldots, x_N)$ is given by
\begin{align}\label{eq_v_n_pointwise}
    \bar{v}^{MC}_{N}(t, x_1, \ldots, x_N) \coloneqq \frac{1}{MN} \sum^M_{m=1} \sum^N_{n=1}  \sum^{T - \Delta t}_{s=t}  f^n_N(s, \hat{X}^{\cdot,m}_s, \alpha^{*}(s, \hat{X}^{\cdot,m}_s)) \Delta_t + g^n_N(\hat{X}^{\cdot,m}_T).
\end{align}
Then, in virtue of \Cref{thm:finite_approx}, given $\mu \in \Pcal_2(\R^d)$, an approximation of the (mean-field) value function $v(t, \mu)$ can be obtained by evaluating the integral
\begin{align}\label{vn mc estimate_integrand}
    \int_{\R^{Nd}}\bar{v}^{MC}_{N}(t, x_1, \ldots, x_N) \ \mu(\de x_1)\ldots \mu(\de x_N).
\end{align}
This provides an approximation for $v(t,\mu)$ for each fixed $(t,\mu)\in [0,T]\times \Pcal_2(\R^d)$. To obtain an approximation of $v$ across the whole domain $[0,T] \times \Pcal_2(\R^d)$, we employ an additional regression step to approximate $\bar{v}_N$, using samples of \eqref{eq_v_n_pointwise} over the finite-dimensional domain $[0,T] \times \R^{Nd}$. Since we do not need to approximate the distribution of $\hat{X}$ in \eqref{eq_v_n_pointwise}, the number of Monte Carlo samples $M$ can be chosen flexibly to accommodate for individual computational resources. This differs from the approach of \cite{pham2022mean, pham2023mean}, where the regression is performed directly for the mean-field value function $v$ on the infinite-dimensional domain $[0,T] \times \Pcal_2(\R^d)$, and a high number of Monte Carlo samples per iteration is required to approximate the distribution of $X$ in \eqref{SDE}.\\

In the case of regression, consider a parametrised function  $\bar{v}_{N,\eta}: [0,T] \times \R^{Nd} \to \R$ with parameters $\eta$. Let $K>0$ denote the number of samples for regression, and let $(t^k, \mathbf{x}^k) \in [0,T] \times \R^{Nd}$, $ 1 \leq k \leq K$. The samples are denoted by $y_1, \ldots, y_K \in \R$, which correspond to the values obtained from the expression \eqref{eq_v_n_pointwise} with the initial conditions $(t^1, \mathbf{x}^1), \ldots, (t^K, \mathbf{x}^K)$ respectively. Consider a mean-squared error (MSE) loss function 
\begin{align*}
    L(\eta) \coloneqq \frac{1}{K} \sum^{K}_{k=1} \lvert \bar{v}_{N,\eta}(t^k, \mathbf{x}^k) - y_k  \rvert^2.
\end{align*}
Let $\pr{\gamma} > 0$ be the regression learning rate, and let $\eta_j$ be the function  parameters at the $j$\textsuperscript{th} iteration, which are updated according to $\eta_{j+1} \coloneqq \eta_j  - \pr{\gamma} \nabla_{\eta}L(\eta)$, for $j \geq 0$. As in the case of the policy iteration step, the iteration terminates when $L(\eta)$ does not decrease after a prespecified number of iterations. The parameters after termination are taking to be approximately optimal and are denoted by $\eta^{*}$. For simplicity, we denote the resulting function $\bar{v}_{N,\eta^*}$ by  $\bar{v}_{N}^{*}$.\\

In practice, we observe a slow convergence rate during the training of $\bar{v}_{N,\eta}$ when minimising for $L(\eta)$. To accelerate the training process, we employ the differential regression method as outlined in \cite{lefebvre2023differential}, which learns in addition the gradient of the value function during training. In particular, each $y_k$ can be considered as a function of the initial condition $(t^k,\mathbf{x}^k)$, so that the derivatives $\partial_t y_k$  and $\nabla_{\mathbf{x}}y_k$ are unbiased estimators of {$\partial_t \bar{v}^{MC}_N(t, x_1, \ldots, x_N)$ and $\nabla_{\mathbf{x}} \bar{v}^{MC}_N(t, x_1, \ldots, x_N)$} respectively. The gradient $\nabla_{\mathbf{x}}y_k$ can in turn be evaluated by computing the partial derivatives of $y_k$ with respect to each simulated trajectory for each player across each $s \in \{0, \Delta t,\ldots, T - \Delta t\}$ via automatic differentiation. Then, $\bar{v}_{N,\eta}$ is trained on the weighted loss function
\begin{align*}
    L_w(\eta) \coloneqq  \frac{1}{K} \sum^{K}_{k=1} &\bigg( w_{\mathrm{val}} \lvert \bar{v}_{N,\eta}(t^k, \mathbf{x}^k) - y_k \rvert^2\\
    &\ + w_{\mathrm{der}} \Big[\lVert \nabla_{\mathbf{x}} \bar{v}_{N,\eta}(t^k, \mathbf{x}^k) - \nabla_{\mathbf{x}} y_k \rVert^2_2
    + \lvert \partial_t \bar{v}_{N,\eta}(t^k, \mathbf{x}^k) - \partial_t  y_k \rvert^2  \Big] \bigg),
\end{align*}
for some constants $w_{\mathrm{val}}, w_{\mathrm{der}} > 0$. We find that for our experiments, optimising over $L_w$ was more effective with shorter experiment times, compared to the approach in \cite{lefebvre2023differential}  of alternating between two separate loss functions of the value function and its derivatives.

\begin{remark}
    When the diffusion coefficient $\sigma^n_N$ is independent of the control $\alpha$, it is also possible to learn an optimal control via a backwards stochastic differential equation (BSDE) characterisation of the value function. In this case, the learning procedure coincides with the one of a BSDE solvers such as in \cite{han2017deep, hpw2021}. 
\end{remark}

\begin{remark}
    The Lipschitz assumptions  \Cref{ass:ex_uniq_SDE,ass:reward} on the coefficients of $b$, $\sigma$, $f$, $g$ are sufficient conditions for the obtaining an unbiased estimator of $\nabla_{\mathbf{x}} \bar{v}_{N,\eta}(t^k, \mathbf{x}^k)$. We refer to the textbook \cite[Ch. 7]{glasserman2004monte} and the references within for a more detailed treatment of pathwise derivative estimators.
\end{remark}

\subsection{Recovering the mean field value function}


Given, for $N\in \N_+$, the approximation $\bar{v}^{*}_{N}$ of the finite-dimensional value function $\bar{v}_{N}$ and recall that the mean field value function is given by the limit in \eqref{eq:main_res}, we obtain its  approximation by evaluating, for $N$ big enough, the integral in \eqref{eq:main_res}. In other words we have, for $N$ sufficiently large,
\begin{align*}
    v(t,\mu) \approx I_N \coloneqq \int_{\R^{Nd}} \bar{v}^{*}_{N}(t,x_1,\ldots,x_N)\ \mu(\de x_1) \ldots \mu(\de x_N).
\end{align*}
\begin{algorithm}[!t]
\label{algo}
        \caption{Finite-dimensional approximation for the MFC problem}
        \begin{algorithmic}[1]
        \STATE{Input: Dynamics $b^n_N$, $\sigma^n_N$; costs $f^n_N$, $g^n_N$; stepsize $\Delta t$; max patience $P_{\rm max}$}        
        \STATE{Input: Control $\alpha_{\theta}$; learning rate $\gamma$}
        \STATE{$v_{\rm best} \leftarrow +\infty$, $patience \leftarrow 0 $}
        \WHILE{$patience < P_{\rm max}$}
        \STATE{Initialise $v = 0$, sample $x = (x_1, \ldots, x_N)$ uniformly on the domain}
        \FOR{$t = 0, \Delta t, \ldots, T-\Delta t$}
        \STATE\label{line_gen_av}{$a \leftarrow \alpha_{\theta}(t, x)$, $v \leftarrow v + \sum^{N}_{n=1} (f^n_N(t, x, a)/N) \Delta t$}
        \STATE\label{line_gen_bm}{Generate Brownian increments $\Delta W^n_t$, $n \in \{1, \ldots,N\}$}
        \STATE\label{line_gen_x}{$x_n \leftarrow x_n + b^n_N(t, x, a) \Delta t + \sigma^n_N(t, x, a) \Delta W^n_t, \ n \in \{1, \ldots ,N\}$ }
        \ENDFOR
        \STATE\label{line_gen_new_v}{$v \leftarrow v + \sum^{N}_{n=1} g^n_N(x)/N$, $\theta \leftarrow \theta - \gamma \nabla_{\theta} v$}
        \IF{$ v < v_{\rm best}$}
        \STATE{$v_{\rm best} \leftarrow v,\ patience \leftarrow 0$}
        \ELSE
        \STATE{$patience \leftarrow patience +1$}
        \ENDIF
        \ENDWHILE
        \RETURN{$\alpha_{\theta^{*}}$}
        \STATE{Input: Learning rate $\pr{\gamma}$; value function $\bar{v}_{N,\eta}$; measure $\mu$}
        \STATE{$L_{\rm best} \leftarrow +\infty$, $patience \leftarrow 0$}
        \WHILE{$patience < P_{\rm max}$}
        \STATE{Initialise loss $L=0$, sample $x = (x_1, \ldots, x_N)$ uniformly on the domain}
        \STATE{Repeat \Cref{line_gen_bm,line_gen_av,line_gen_x,line_gen_new_v} to compute $v$, but replace $\alpha_{\theta}$ with $\alpha_{\theta^*}$.}
        \STATE{$L \leftarrow L +  \lvert \bar{v}_{N,\eta}(0,x_1, \ldots, x_N) - v \rvert^2 $, $\eta \leftarrow \eta - \pr{\gamma} \nabla_{\eta}L $}
          \IF{$ L < L_{\rm best}$}
        \STATE{$L_{\rm best} \leftarrow L,\ patience \leftarrow 0$}
        \ELSE
        \STATE{$patience \leftarrow patience +1$}
        \ENDIF
        \ENDWHILE
        \RETURN{$\int_{x_1,\ldots, x_N} \bar{v}^*_{N}(0,x_1,\ldots, x_N)\ \mu(x_1) \ldots \mu(x_N)$}
        \end{algorithmic}
        \end{algorithm}
In the case where the process $X$ is real-valued, the integral $I_N$ can be evaluated by using the inverse transform sampling. Given $\mu \in \Pcal(\R)$, let ${F_{\mu}:\R \to [0,1]}$ be the cumulative distribution function associated to $\mu$. Then one obtains
\begin{align*}
    I_N = \E \left[ \bar{v}^{*}_{N}(t, F_{\mu}^{-1}(U_1),\ldots,F_{\mu}^{-1}(U_N)) \right],
\end{align*}
where $U_1,\ldots, U_N$ are independent uniformly distributed random variables over $[0,1]$, so that evaluating $I_N$ is reduced to sampling over the unit hypercube. This allows for the required number of samples to be reduced by the use of quasi-Monte Carlo methods, for example employing the use of Sobol sequences. We summarise the overall algorithm in this section in pseudo-code in \Cref{algo} for the MSE loss function $L$. Observe that the first ``\text{while}'' loop in \Cref{algo} learns the optimal control, while the second loop learns the finite-dimensional value function.

\section{Numerical experiments}\label{sec:experiments}
In this section we consider three experiments for our proposed numerical algorithm. The rate of convergence of the finite dimensional value functions towards the mean field value function is only proven for a few specific cases \cite{germain2022rate,carmona2022convergence,cardaliaguet_regularity_2023,cecchin2024quantitative}, and to our best knowledge, the general case remains an open problem. Our numerical tests suggests that an upper bound of $N=500$ players can approximate the mean field value function up to a 0.1\% degree of accuracy. In the case of a weaker mean field coupling, it maybe be possible to consider a lower dimension for a similar level of accuracy, such as the case in Example 3, where we perform the regression step with 100 players. This reduces the complexity of the problem compared to the sampling and regression of the value function in the Wasserstein space.\\

In all experiments, the training is performed on the Google Colab platform, using the NVIDIA T4 GPU. We choose our examples such that explicit solutions are available for benchmarking purposes. In these cases, we evaluate the performance of the method by computing a residual loss for the mean field control problem. Let $P>0$ be the number of samples in a test set of quantized measures, at which that value functions will be evaluated, and let $L>0$ be the number of sample points on the underlying domain $\R^{Nd}$. The quantized measures can be generated as follows. Sample a number of points $x_1,\ldots, x_L \in \R^{Nd}$, and sample the exponential random variables $E_1,\ldots, E_L \sim \Exp(1)$, set $\pi_l = E_l/(\sum_j E_j)$. For $1 \leq p \leq P$, define the quantized measures
$\mu^p = \sum^L_{l=1} \pi_l \delta_{x_l}$. Then, we compute the residual loss function
\begin{align*}
    \mathcal{L}_{res} \coloneqq \frac{1}{P} \sum^P_{p=1} \left\lvert v(t,\mu^p) - \int_{\R^{Nd}} \bar{v}^{*}_{N}(t,x_1,\ldots, x_N)\ \mu^p(\de x_1) 
              \ldots \mu^p(\de x_N) \right\rvert^2.
\end{align*}
When an analytical expression for $v(t,\mu^p)$ is not available, we can still evaluate the performance of the method by considering the accuracy of the approximation $\bar{v}^{*}_{N}$ for the finite-dimensional value function. In this case, we compute the loss associated to the finite-dimensional Hamilton-Jacobi-Bellman (HJB) equation as follows. Let $\Delta t ,\Delta x \in \R_+$ be given step sizes in the time and spatial directions respectively. Let $\mathcal{T} = \{0, \Delta t, \ldots, T- \Delta t, T\}$, and let $\mathcal{X}$ be a discretised grid with step $\Delta x$ on the domain $\R^d$, on which the computation is performed. For a sufficiently smooth function $\phi: [0,T] \times \R^{Nd} \to \R$, we denote by $\D_t(\phi)$ the approximation of the partial derivative $\partial_t \phi$ obtained by automatic differentiation. Similarly, we denote by $\D_{\mathbf{x}}(\phi)$ and $\D^2_{\mathbf{x}}(\phi)$ the approximation of the gradient vector $D_{\mathbf{x}}(\phi)$ and Hessian matrix $D^2_{\mathbf{x}}(\phi)$ respectively.
Then, consider the HJB loss
\begin{align*}
\mathcal{L}_{\text {HJB}}:= & \frac{1}{|\mathcal{T}||\mathcal{X}|} \sum_{\substack{t \in \mathcal{T}\\ \mathbf{x} \in \mathcal{X} }} \left|\D_t(\bar{v}^{*}_{N})(t,\mathbf{x}) + H \big(t,\mathbf{x}, \D_{\mathbf{x}}(\bar{v}^{*}_{N})(t,\mathbf{x}), \D^2_{\mathbf{x}}(\bar{v}^{*}_{N})(t,\mathbf{x}) \big)\right|^2 \\
\qquad & + \frac{1}{|\mathcal{X}|} \sum_{\mathbf{x} \in \mathcal{X}}\left|\bar{v}^{*}_{N}(T,\mathbf{x})-\frac{g^n_N(\mathbf{x})}{N}\right|^2,
\end{align*}
where $H:[0,T] \times \R^{Nd} \times \R^{Nd} \times\R^{(Nd \times Nd)}  \to \R$ is the Hamiltonian for the associated finite-dimensional PDE, i.e.,
\begin{align*}
    H(t, \mathbf{x}, p, M) \coloneqq& \frac{1}{N} \sum^N_{n= 1} f^n_N (t,\mathbf{x}, \alpha^{*}_{\theta}(t,\mathbf{x})) + \sum^N_{n= 1}\langle   b^n_N(t, \mathbf{x}, \alpha^{*}_{\theta}(t,\mathbf{x})) , p^n \rangle\\
    &+ \frac{1}{2}\sum^N_{n= 1} \sum^N_{\pr{n}=1} \ \tr\left( \sigma^n_N(t, \mathbf{x}, \alpha^{*}_{\theta}(t,\mathbf{x})) \sigma^{\pr{n}}_N(t,\mathbf{x}, \alpha^{*}_{\theta}(t,\mathbf{x}))^{\top} M_{\pr{n},n} \right)
\end{align*}
where $p^n$ contains the entries $(p_{n}, \ldots, p_{n+d-1})$ of the vector $p$, and $M_{\pr{n},n}$ is the submatrix of $M$ containing the rows from $\pr{n}$ to $\pr{n} + d - 1$ and the columns from $n$ to $n + d -1$ inclusive. Finally, for illustration purposes, we will also plot the values of $\bar{v}^*_{N}$ against Monte Carlo estimates, generated by the learned control $\alpha^{*}$, along cross-sections or curves in the computed domain.

\subsection{Example 1}
 Suppose $W = (W_t)_{t\in[0,T]}$ is a real-valued Brownian motion on a probability space $(\Omega, \F, \PP)$. Let $X^{ 0 , \xi,\alpha} = (X^{ 0 ,\xi,\alpha}_t)_{t \in [0,T]}$ be a real-valued controlled stochastic process, solution to the SDE
\begin{align*}
    \de X_t = \alpha_t\, \de t +  \de W_t,\quad X_0 = \xi\sim \mu,
\end{align*}
where $\alpha$ is a real-valued progressively measurable process valued in $A\subseteq \R$. 
Let $m_0\in\R$, $\sigma_0\in\R_+$ and consider the following value function
\begin{align}\label{eq:value_test_3}
    v(0,\mu):= \inf_{\alpha\in\mathcal{A}}\bigg\{\Big(\E[X^{ 0 ,\xi,\alpha}_T]-m_0\Big)^2 + \Big(\mathrm{Var}(X^{ 0 ,\xi,\alpha}_T)-\sigma_0^2\Big)^2\bigg\}.
\end{align}
\begin{remark}
Notice that $v(0,\mu)\geq 0$ and a sufficient condition for the existence of an optimal control is that there exists a control $\alpha\in\mathcal{A}$ such that $\E[X^{{0 },\xi,\alpha}_T]=m_0$ and $\mathrm{Var}(X^{{0 },\xi,\alpha}_T)=\sigma_0^2$. Consequently, when looking at solving the minimisation problem \eqref{eq:value_test_3}, $m_0$ and $\sigma_0$ can be interpreted, respectively, as a target mean and standard deviation for the terminal state. In this context, the cost function plays a role as a penalty function on the distribution of the state at terminal time. See also, for example, the soft-constrained Schr\"odinger bridge problem in \cite{ma2025schr}.
\end{remark}  
\begin{figure}[!t]
    \centering
    \includegraphics[width = 0.9\textwidth]{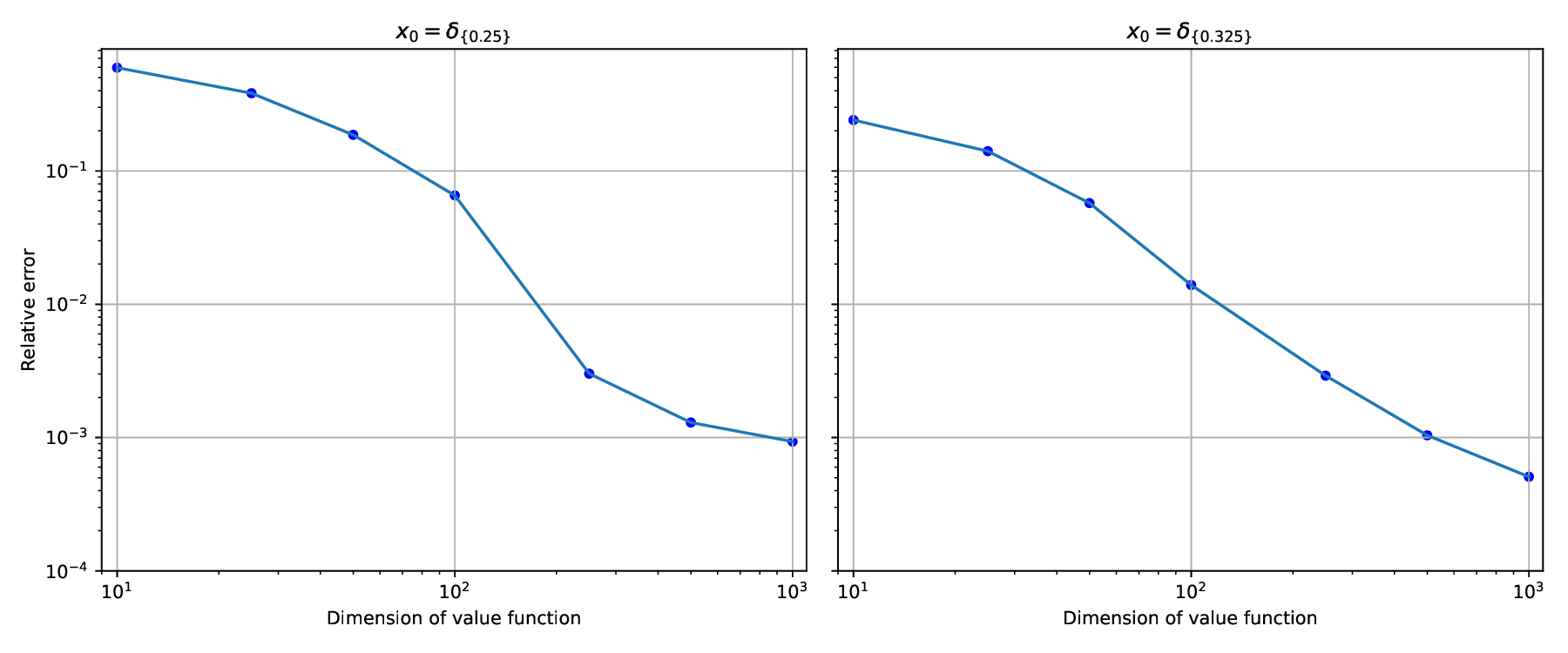}\hfill
    \includegraphics[width = 1\textwidth]{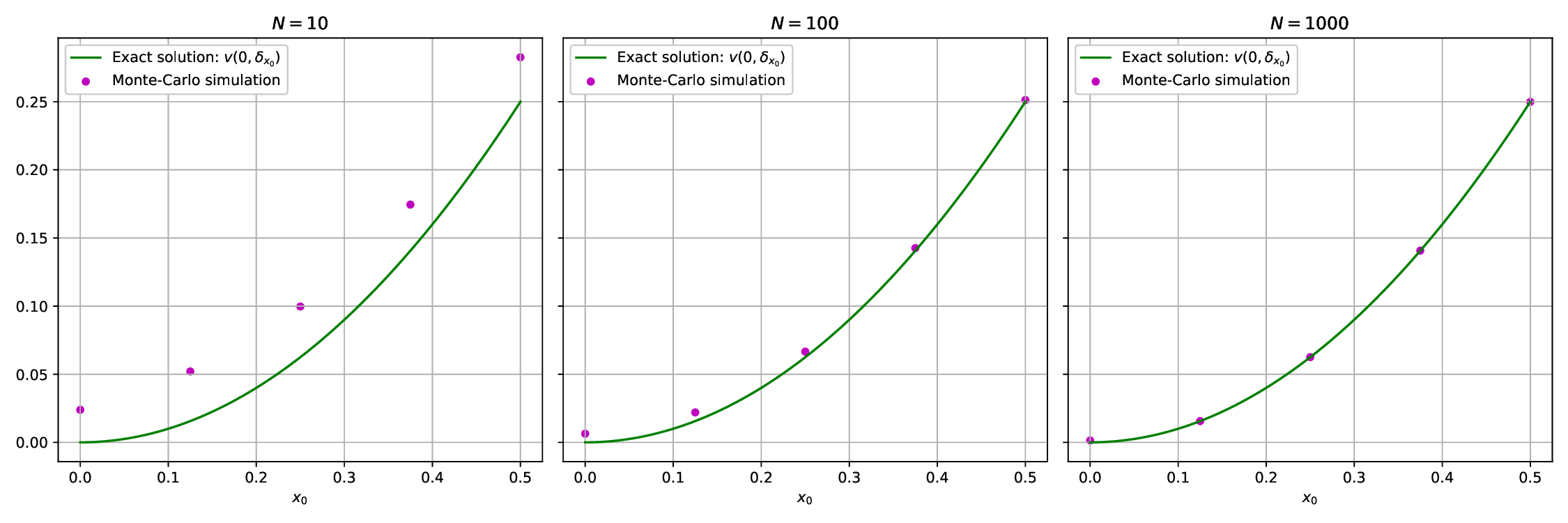}
    \caption{(Example 1) Top: empirical convergence of the finite-dimensional value function as the number of particles $N\to\infty$. Bottom: exact mean field solution $v(0,\delta_{x_0})$ versus generated Monte Carlo estimates, for $x_0\in\{0.000, 0.125, 0.250, 0.375, 0.500\}$, number of players $N$ increasing.}
    \label{fig:III}
\end{figure}
Let $m_0=0$ and $\sigma_0=1$. Having 
$
\E[X^{{0 },\xi,\alpha}_T] = \E[\xi] + \E[\int_0^T \alpha_t\ \de t],
$
one easily gets
\begin{align}\label{value_ex1}
    v(0,\mu) &\geq \inf_{\alpha\in\mathcal{A}}\bigg\{\E\bigg[\xi + \int_0^T \alpha_t\ \de t\bigg]^2\bigg\}
    \geq \E[\xi]^2 + 2\inf_{\alpha\in\mathcal{A}}\bigg\{\E[\xi]\ \E\bigg[\int_0^T \alpha_t\ \de t\bigg]\bigg\}.
\end{align}
Assuming, for instance, $\E[\xi]\geq0$ and $A\subseteq [0,+\infty)$ we deduce
\begin{align}\label{eq:3ex_ineq}
    v(0,\mu) \geq \E[\xi]^2.
\end{align}
It can be easily checked that, if $T=1$ and $\xi\sim\delta_{x_0}$, with $x_0\geq0$, taking $\hat \alpha\equiv 0$, the lower bound $\E[\xi]^2$ is achieved. Indeed, one has  $\E[X^{{0 },\xi,\hat \alpha}_T]=x_0$ and $\mathrm{Var}(X^{{0 },\xi,\hat \alpha}_T)=1$, ending with $J(0,\xi,\hat \alpha) = x_0^2 = \E[\xi]^2 $ and then $v(0,\delta_{x_0})=x_0^2$.\\

We consider hereafter the control constrained into $A=[0,1]$. From the numerical point of view, the constraint has been introduced with the \texttt{sigmoid} as activation function $\sigma_{\rm out}:\R \to (0,1)$ for the last layer of the neural network learning the control $\alpha^*$, i.e. 
$\sigma_{\rm out}(x) = 1/(1+ e^{-x}).$
It easy to check that minimizing over controls valued on $[0,1]$ in \eqref{value_ex1} gives the same value as taking $(0,1)$-valued controls.
It is indeed sufficient to consider the sequence of controls $\alpha_j\equiv\frac{1}{j}$, for $j\in\mathbb{N}_{+}$, to get
\begin{align*}
    \E[X^{{0}, \xi,\alpha_j}_T] = \E[\xi]+\frac{T}{j} = x_0 + \frac{1}{j} \quad\text{and}\quad \mathrm{Var}(X^{{0}, \xi,\alpha}_T) = \mathrm{Var}(\xi) + T = 1,
\end{align*}
then
\begin{align*}
   J\Big(0,\delta_{x_0},\alpha_j\Big) =  \bigg(x_0+\frac{1}{j}\bigg)^2 + (1-1)^2 = x_0^2 + \frac{1}{j^2}+2x_0\frac{1}{j} \xrightarrow[]{j\to+\infty} x_0^2 = \E[\xi]^2,
\end{align*}
which combined with \Cref{eq:3ex_ineq}, implies $v(0,\delta_{x_0})=\E[\xi]^2$. \\
\begin{figure}[!t]
    \centering
    \includegraphics[width = 0.7\textwidth]{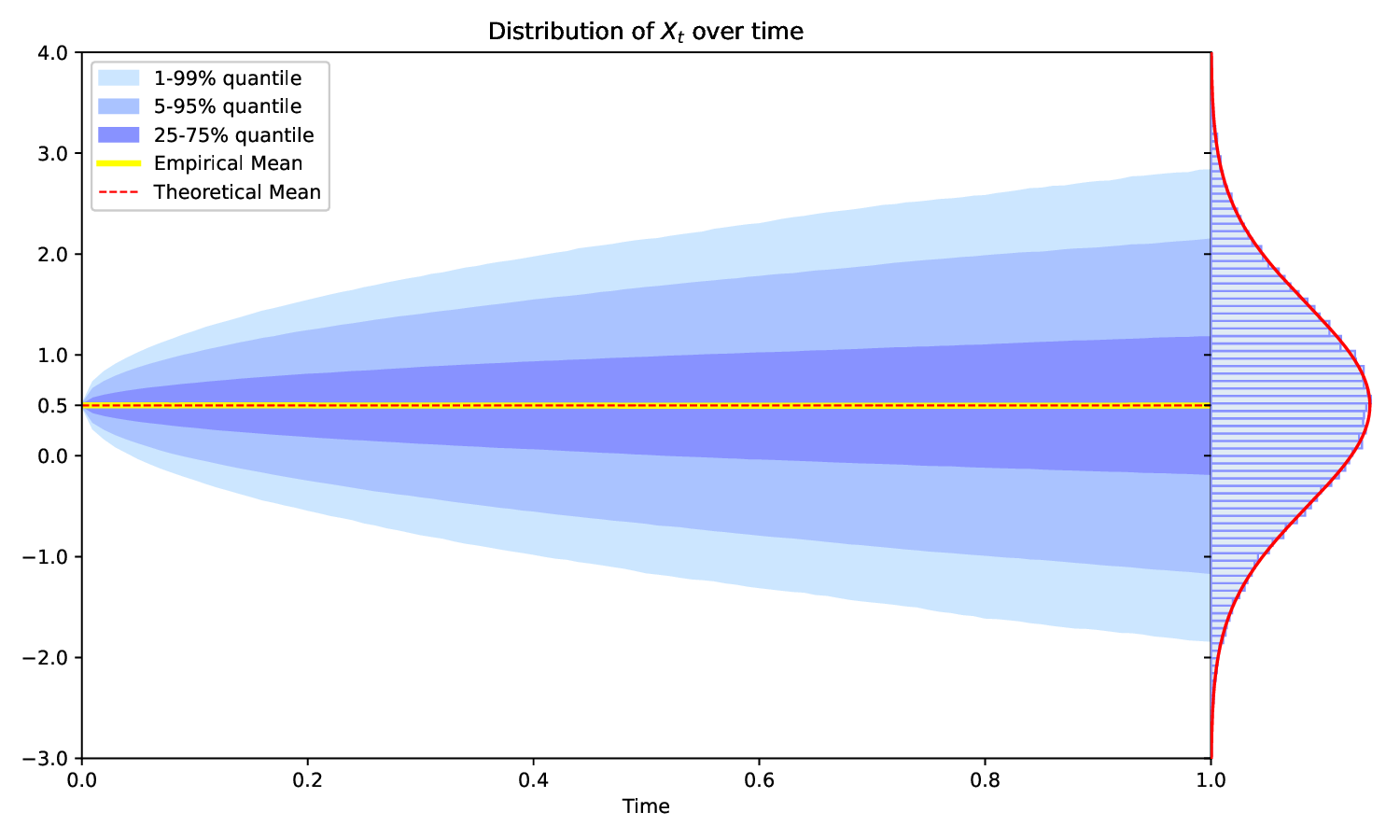}
    \caption{(Example 1) Plot of the distribution quantiles over time, for one simulated player of the process $X^n_t$, starting at $\xi\sim\delta_{0.5}$. The histogram of $X^n_T$ at the terminal time $T=1$ compared with theoretical density $\mathcal{N}(0.5,1)$ is plotted simultaneously on the right vertical axis. }
    \label{fig:III_ex_distrib}
\end{figure}

For our experiments, we trained the neural network for $\bar{v}_N(\cdot,0)$ for points uniformly distributed in $[-1,1]$. The above benchmark of the method is provided to be $v(0,\delta_{x_0})=x_0^2$, and we plot the absolute relative error for different values of $N$ within $10$ and $10^3$ and using 10,000 trajectories for the simulation. The top graph of \Cref{fig:III} demonstrates empirical convergence of our method through the exact mean field value function, for two different initial points. Moreover, as previously shown in \Cref{eq:3ex_ineq}, $\E[\xi]^2$ represents a lower bound for $J(0,\xi,\alpha)$, for every admissible constrained control $\alpha$ in $[0,1]$. A graphical representation for $\xi\sim\delta_{x_0}$, i.e. ${J(0,\delta_{x_0},\alpha) \geq x_0^2}$, is provided in the bottom graph of \Cref{fig:III} (unless the Monte Carlo error is involved), along with approximation improvement with increasing $N$. Finally, in \Cref{fig:III_ex_distrib} we plot the quantiles of the distribution of the trajectory $X^n_t$ for a single player. The plot resembles a Gaussian distribution, for each time step, and this has been verified by testing the assumption with the  Pearson normality test (\texttt{scipy.stats.normaltest}, returning an average $\texttt{p-value}=0.5753$ over time, with a minimum value above $0.05$, sampling $100,000$ trajectories). This is what we expect, taking the optimal control $\hat\alpha\equiv 0$, as  for all  $s\in[0,1], X^{0,x_0,\hat \alpha}_s = x_0 + W_s \sim \mathcal N(x_0, s)$. 


\subsection{Example 2}

We consider an optimal liquidation problem with permanent price impact as described in  \cite{acciaio2019extended}. A group of investors aim to liquidate their position by controlling their trading speed over time. This generates a permanent market impact as the price is assumed to depend on the average trading speed of the investors. By considering a continuum of investors, the inventory process is given by
\begin{align*}
    \de Q_t = \alpha_t\, \de t,\quad Q_0=\xi \sim \mu_0 \in\Pcal_2(\R).
\end{align*}
Here $\alpha_t$ generally takes on negative values as investors are liquidating their portfolio. The wealth process is given by
\begin{align*}
    \de X_t =  - \alpha_t(S_t + k\alpha_t)\, \de t,\quad  X_0 = 0,
\end{align*}
where the parameter $k$ represents the temporary market impact which only affects the individual investor. The price process $S$ follows the dynamics
\begin{align*}
    \de S_t = \lambda \E [\alpha_t]\, \de t + \sigma\, \de W_t,\quad S_0 = s_0,
\end{align*}
where the parameter $\lambda \geq 0$ represents the permanent price impact. To ease the notation we will write $(S, Q, X)=(S^{0, s_0,\alpha}, Q^{0, \xi,\alpha}, X^{0, 0,\alpha})$. The cost functional is 
\begin{align*}
    \E \left[- X_T  - Q_T (S_T  - \psi Q_T ) + \phi \int^T_0 Q^2_t\, \de t \right] ,
\end{align*}
where $\psi$ is a penalty parameter for excess inventory at terminal time, and $\phi$ is an `urgency' parameter for the running cost. By rewriting the cost functional as 
\begin{align*}
    \E \left[ \int^T_0 (\alpha_t S_t + k\alpha^2_t + \phi Q^2_t )\, \de t  - Q_T (S_T  - \psi Q_T )\right],
\end{align*}
the problem can then be considered as an extended MFC problem, with a two-dimensional state process $(S,Q)$ and value function
\begin{align*}
    v(0,\mu) = \inf_{\alpha \in \A} \E \left[ \int^T_0 (\alpha_t S_t + k\alpha^2_t + \phi Q^2_t )\, \de t  - Q_T (S_T  - AQ_T )\right],\ \mu = \delta_{s_0} \otimes \mu_0 .
\end{align*}

Note that the diffusion coefficient of $(S,Q)$ in this problem is degenerate. It is shown in \cite{acciaio2019extended} that the explicit solution of the optimal control is
\begin{align}\label{opt_liquid_sol}
    \hat{\alpha}_t = Q_0 r \frac{d_1 e^{-r(T-t)}  - d_2 e^{r(T- t)}}{d_1 e^{-rT} + d_2 e^{rT}} + \E [Q_0] \frac{2\lambda \phi (e^{-rt} + e^{rt})}{(d_1e^{-rT} + d_2e^{rT} )(c_1e^{-rT} + c_2e^{rT} ) },
\end{align}
with constants $d_1 = \sqrt{\phi k}  - \psi$, $d_2 = \sqrt{\phi k} + \psi$, $c_1 = 2d_1 + \lambda$, $c_2 = 2d_2  - \lambda$ and $r = \sqrt{\phi/k}$. In particular, when the initial condition $Q_0$ is non-deterministic, the optimal control $\hat\alpha_t$ is not in a closed loop form. As outlined in \Cref{sigma_degen}, we shall regularise the problem by considering the system
\begin{align*}
    \de S_t = \lambda \E [\alpha_t]\, \de t + \sigma\, \de W_t, \quad
     \de Q^{\varepsilon}_t = \alpha_t\, \de t + \varepsilon \ \de B_t,
\end{align*}
where $\varepsilon>0$ is small and $B$ is independent of $W$, so that we work towards an approximation of $\bar v_{\varepsilon, N}$.\\
\begin{figure}[!t]
    \centering
    \includegraphics[width = 0.5\textwidth]{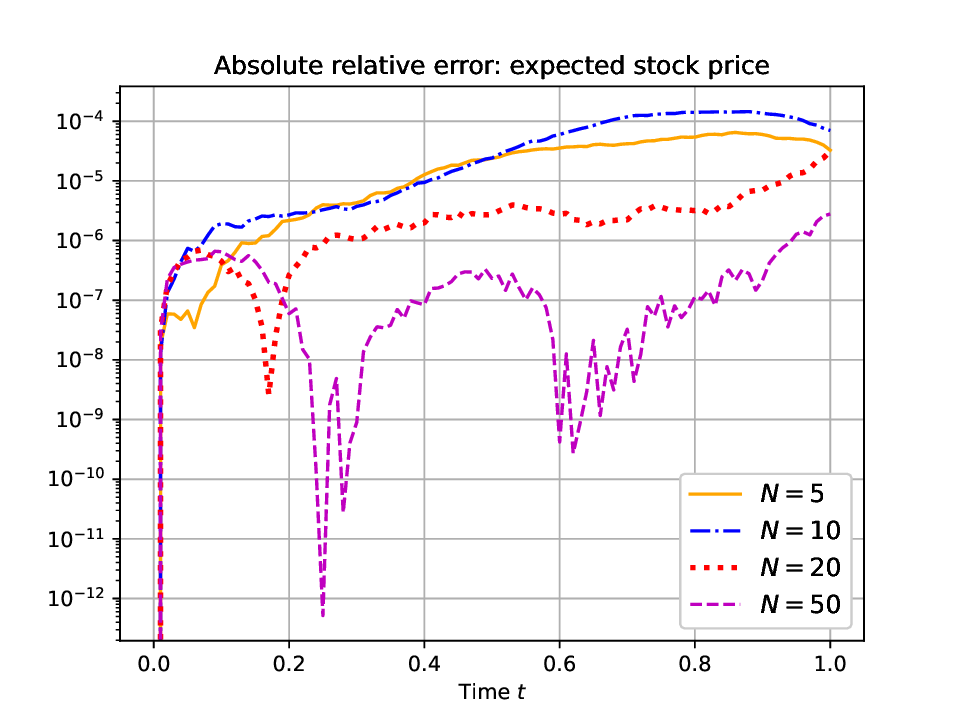}\hfill
    \includegraphics[width = 0.5\textwidth]{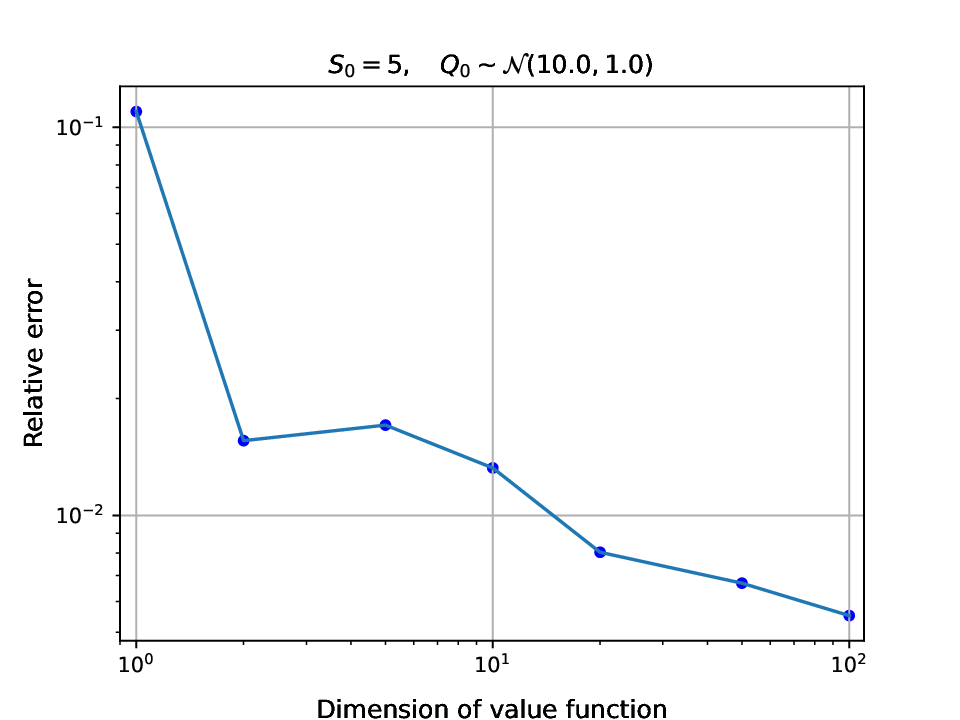}
    \caption{(Example 2) Left: The absolute relative error of the expected stock price, comparing the explicit solution with the finite-dimensional approximation. Right: empirical convergence of the finite-dimensional value function as the number of particles $N\to\infty$.}
    \label{fig:price_impact_dim}
\end{figure}

For our experiments, we consider an initial condition of $S_0 = 5$ and $Q_0 \sim \mathcal{N}(10, 1)$. We also take the following parameters: $T=1$, $k = 0.2$, $\lambda = 0.4$, $\phi = 0.1$, $\psi=2.5$, $\sigma = 1$, and $\varepsilon = 0.01$ for the regularisation parameter. In order to benchmark our method, we utilise the optimal control $\hat \alpha_t$, explicitly given in \eqref{opt_liquid_sol}, to generate Monte Carlo estimates for the corresponding state process $(\hat S, \hat Q)$ and the value function $v$. We first consider the trajectory of the expected stock price $t \mapsto \E[\hat S_t]$, and compare it with the trajectory of the empirical mean of the finite dimensional approximation $t \mapsto \frac{1}{N}\sum^N_{n=1} \E[S^n_t]$. We plot the absolute relative error of the two trajectories on the left graph of \Cref{fig:price_impact_dim} for different values of $n$, using 500,000 trajectories for the simulation to minimise the Monte Carlo error. 
On the right graph, we illustrate an instance of the decay of the relative error of the approximated value function against the mean field value function as the value of $n$ increases, for a fixed value of $\mu = \Ncal(10,1)$. 
By considering the regularised problem, we obtain a feedback control, parametrised in the form of a neural network, that approximates an optimal control of general open loop form.

\subsection{Example 3} 
Suppose $W = (W_s)_{s\in[0,T]}$ is a real-valued Brownian motion on a probability space $(\Omega, \F, \PP)$. Let $X = (X_s)_{s \in [t,T]}$ satisfy the SDE
\begin{align*}
    \de X_s = \alpha_s\, \de s + \sqrt{2}\, \de W_s,\quad \mathcal{L}({X_t}) = \mu,
\end{align*}
where $\alpha$ is a real-valued progressively measurable process. Given a terminal cost function $g: \R \to \R$, consider the value function 
\begin{align*}
    v(t, \mu) = \inf_{\alpha \in \A} \left( \int^T_t \E^{\mu}\left[\frac{\lvert \alpha_s \rvert ^2}{2}\right]  \de s + \E^{\mu} \left[ g(X_T) \right] \right),
\end{align*}
where $\E^{\mu}$ denotes the expectation under which $\Lcal({X_t}) = \mu$. It can be shown that $v$ satisfies the following master Bellman PDE (see \cite[Section 2.4]{daudinseeger2024})), on  $[0, T) \times \mathcal{P}_2(\mathbb{R})$,
\begin{align}\label{eq:ex1master}
-\partial_t v(t,\mu) -  \int_{\R^{d}} \left( \text{div}_x \partial_\mu v(t,\mu)(x) - \frac12 |\partial_\mu v(t, \mu)(x)|^2\right) \mu(\de x)  =0,
\end{align}
with terminal condition $v(T, \mu)  =\int_{\R} g(x)\, \mu(\de x)$.
Such a PDE has a unique, classical solution given by
\[
    v(t,\mu) = \int_{\R} w(t,x)\, \mu(\de x),
\]
where ${w:[0,T]\times  \R\to \R}$ solves the following one-dimensional HJB equation:
\begin{align*}
\begin{dcases}
-\partial_t w(t,x) -\Delta w (t,x) + \frac{1}{2}|\nabla w(t,x)|^2 =0,\ & (t,x) \in [0, T) \times \mathbb{R}, \\
w(T,x) = g(x),\ & x \in \mathbb{R},
\end{dcases}
\end{align*}
which has explicit solution 
\begin{align*}
    w(t,x) = -2 \log\left(\frac{1}{\sqrt{4 \pi(T-t)}} \int_{\R} \exp \left( -\frac{(x-y)^2}{4(T-t)} - \frac{g(y) }{2}\right) \de y\right).
\end{align*}
In this context, the argument of the limit in \Cref{thm:n_approx} is constant with respect to the number of players $N$, and one has the equality for any $N \in \N_+$:
\begin{align*}
    v(t,\mu) =  \int_{\R^{N}} \frac{1}{N} \sum^N_{n=1}w(t,x_n) \, \mu(\de x_1) \otimes \ldots \otimes \mu(\de x_N).
\end{align*}
In this specific case, due to the integral form of the exact solution of the problem, we can test the validity of our numerical approximation by simply examining the accuracy of the  finite-dimensional solver. In fact, it is straightforward to verify that 
\begin{align*}
    \frac{1}{N}\sum^N_{n=1}w(t,x_n) = v_N(t,\mathbf{x}),\quad \mathbf{x} = (x_1, \ldots, x_N) \in \R^N,
\end{align*}
where $v_N: [0,T] \times \R^N \to \R$ is the finite-dimensional approximation of \eqref{eq:ex1master}, introduced in Section \ref{sec:finitedimapprox}, which is a solution the following HJB equation on $[0, T) \times \mathbb{R}^{N}$,
\begin{align}\label{eq:pde_3ex}
-\partial_t u(t, \mathbf{x}) -\sum^N_{n=1} \Delta_{x_n} u(t, \mathbf{x}) + \inf _{a \in \mathbb R^N} \left\{  \sum^N_{n=1} \left(a_n \cdot \nabla_{x_n} u(t, \mathbf{x}) + \frac{|a_n|^2}{2N}  \right)   \right\}   =0.
\end{align}
with terminal condition $u(T, \mathbf{x}) = \bar{g}_N(\mathbf{x})\coloneqq \frac{1}{N}\sum^N_{n=1}g(x_n)$. Eq. \eqref{eq:pde_3ex} simplifies to
\begin{align*}
    -\partial_t u -\Delta u + \frac{N}{2}|\nabla u|^2 & =0, \text { on }[0, T) \times \mathbb{R}^{N}.
\end{align*}

In our experiments, we take the terminal condition as the function $g(x) = \lvert x\rvert^2 / {2}$. In order to satisfy the boundedness assumption of $g$ in \Cref{ass:reward}, we truncate $g$ as necessary outside some fixed domain. This aligns with the practical experiments, since the computational step is always confined within a bounded grid.\\

\begin{figure}[!t]
    \centering
    \includegraphics[width = 0.45 \textwidth]{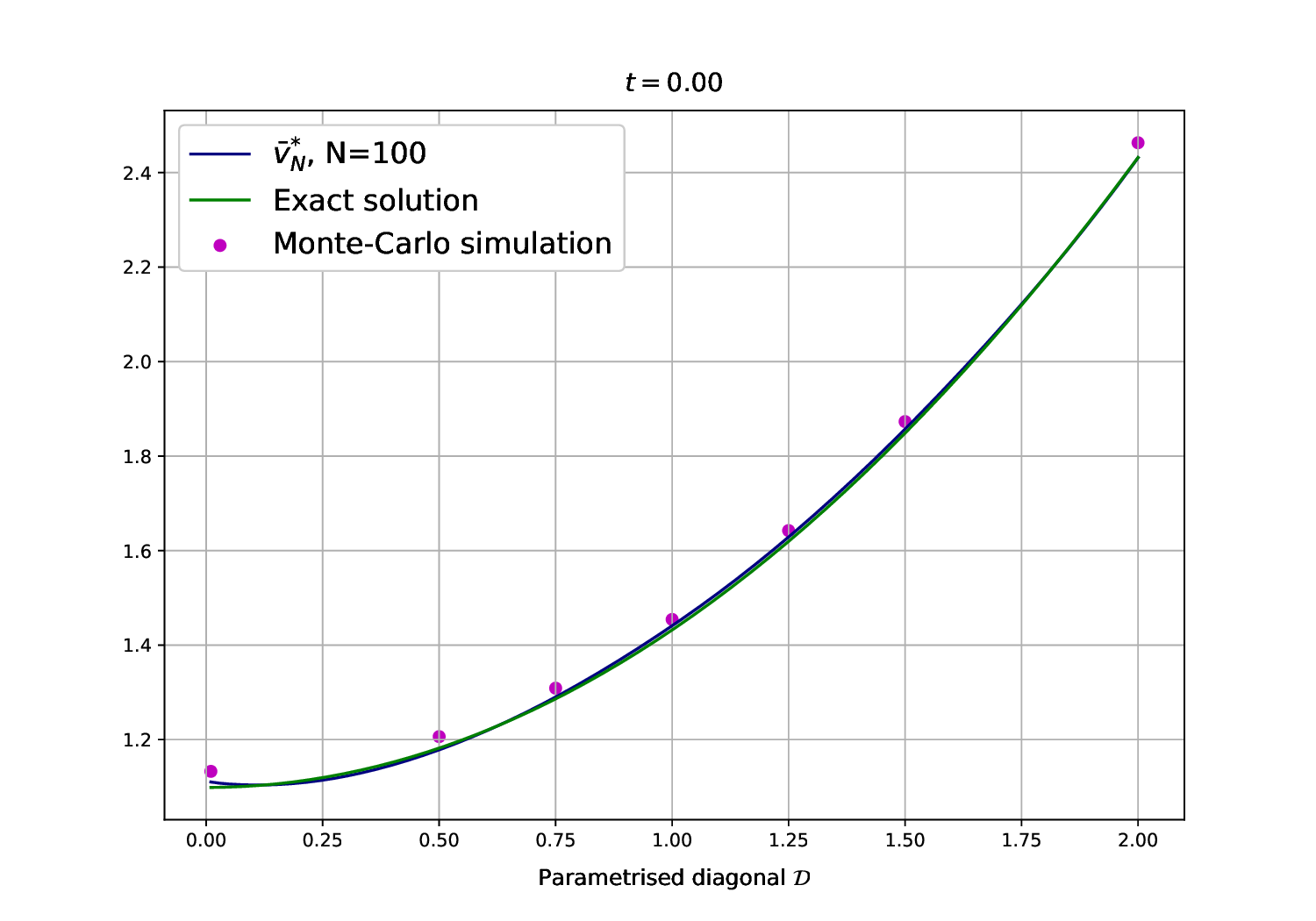}
    \includegraphics[width = 0.45 \textwidth]{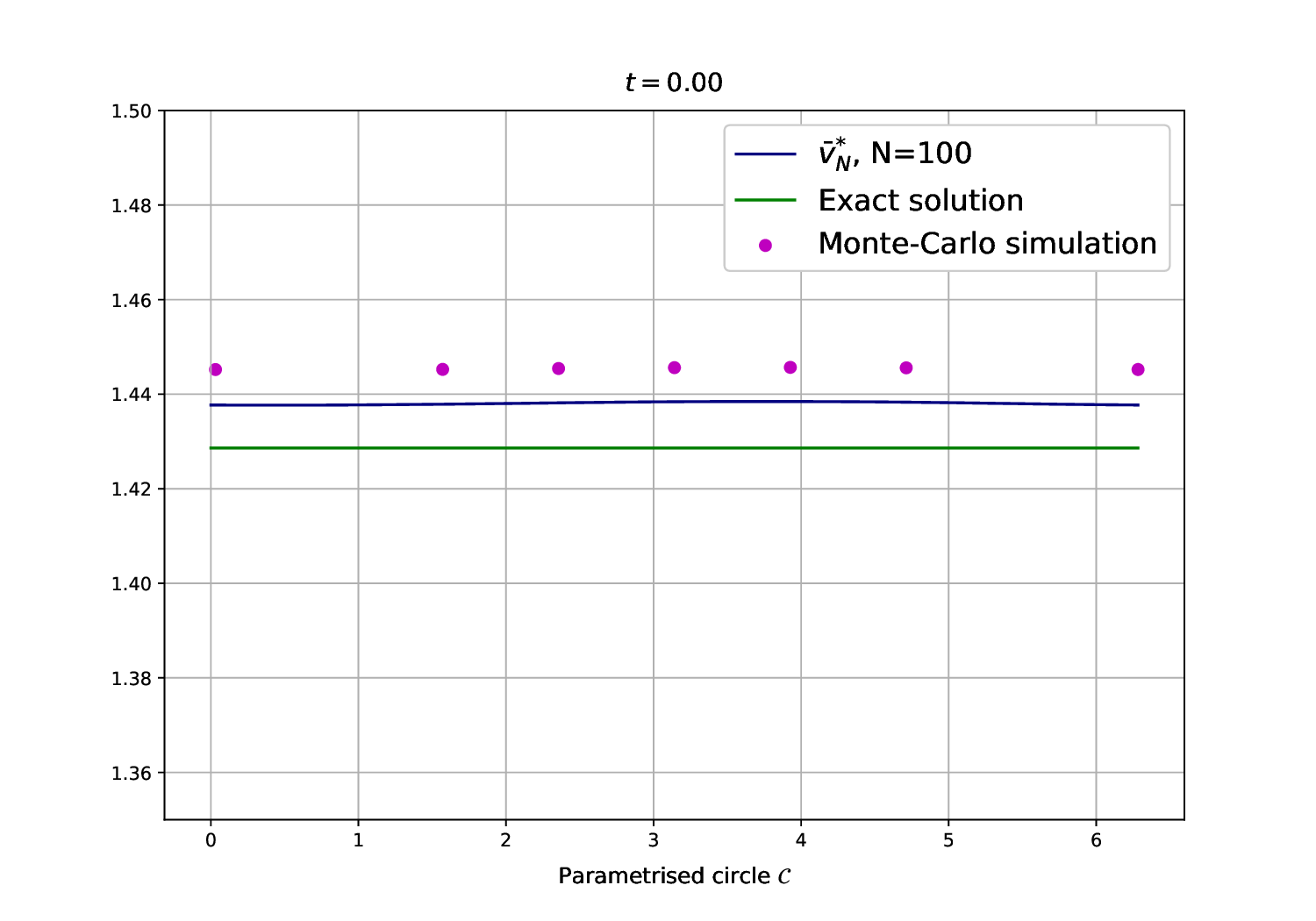}
    \includegraphics[width = 0.45 \textwidth]{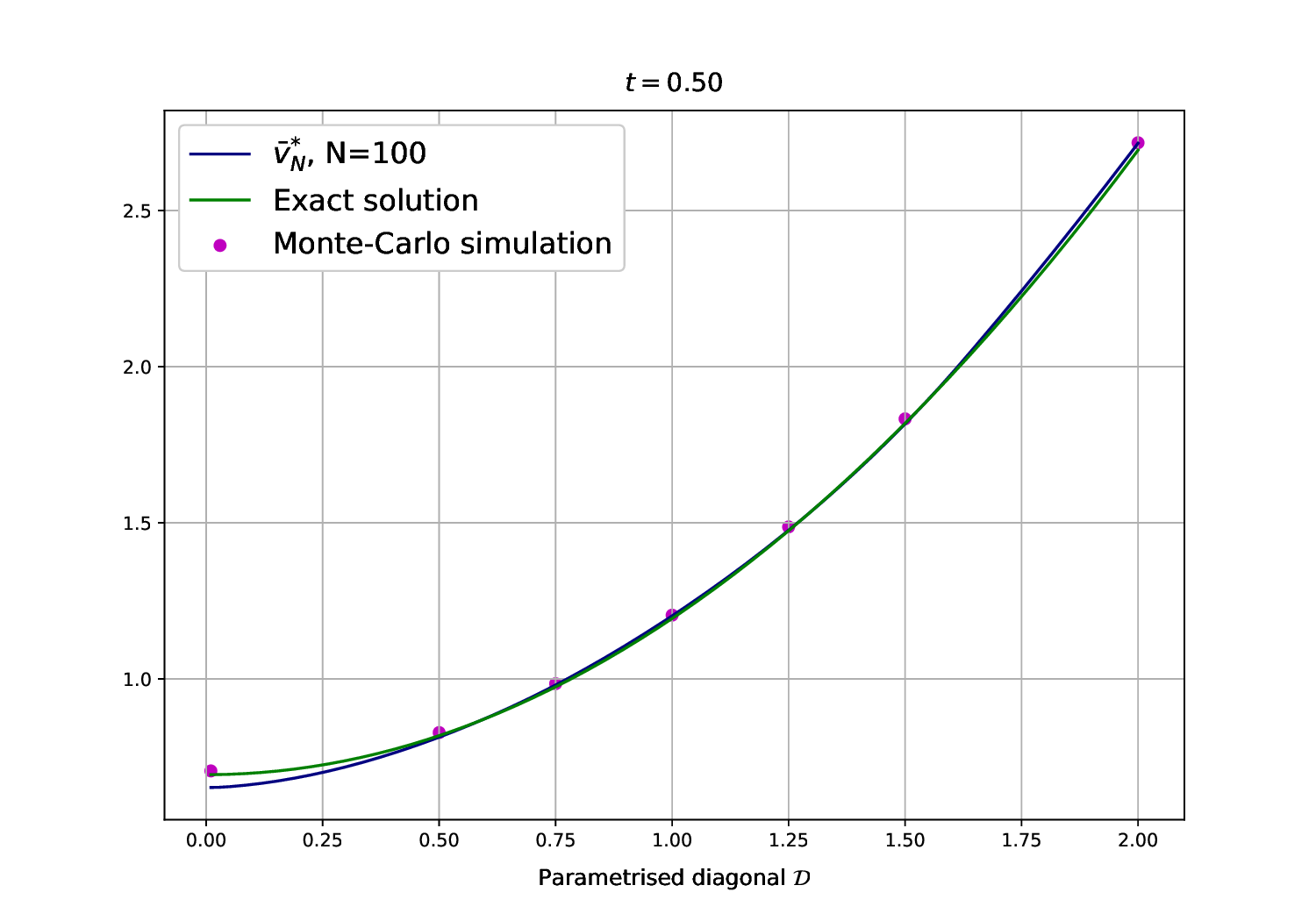}
    \includegraphics[width = 0.45 \textwidth]{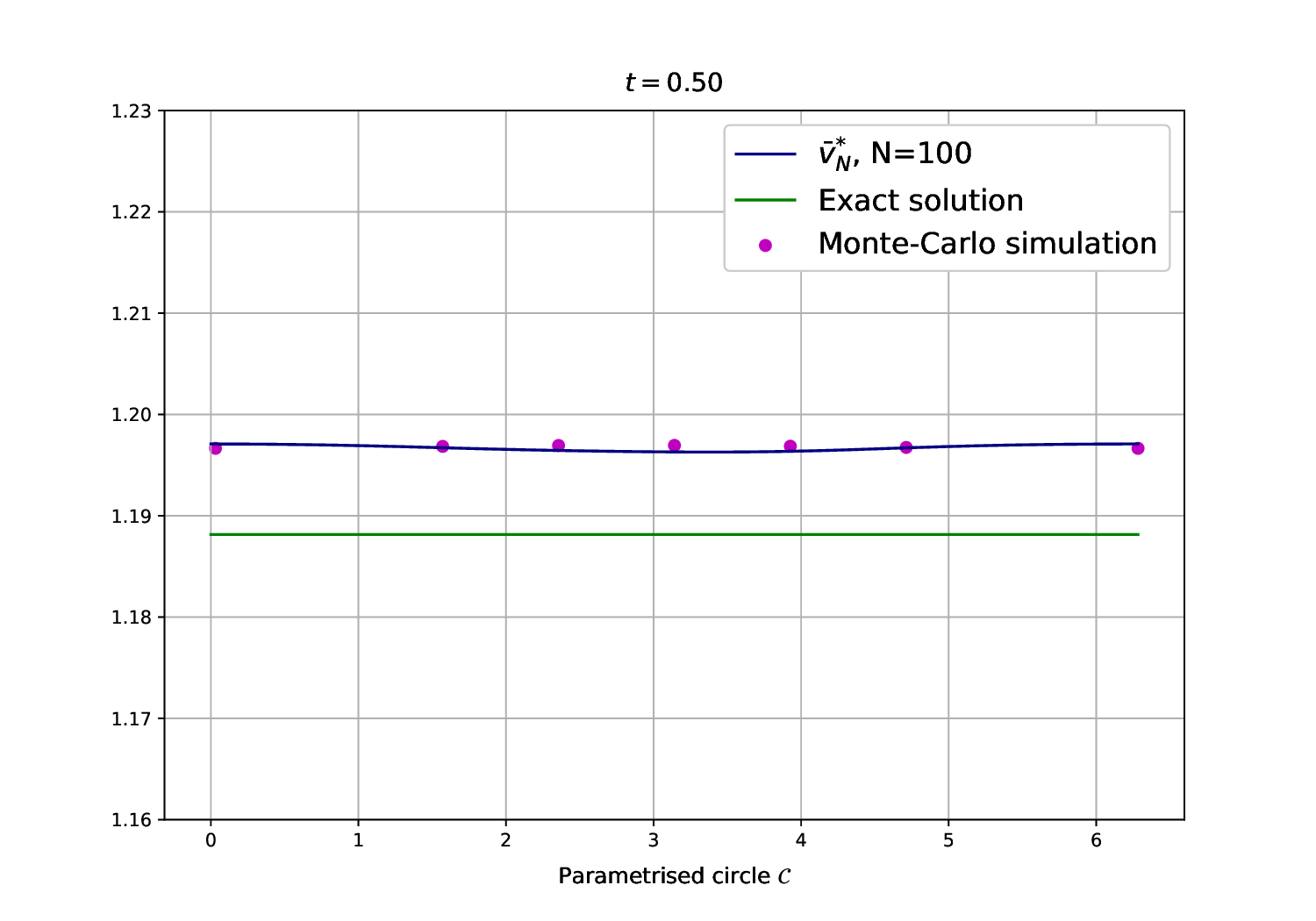}
    \caption{(Example 3) Left column: Tracing $\bar{v}^{*}_{N,\eta}$ along the diagonal of the hypercube $[0,2]^{100}$. Right column: Tracing $\bar{v}^{*}_{N,\eta}$ along a circle on the first two dimensions.}
    \label{fig:daudin_ex}
\end{figure}

\Cref{fig:daudin_ex} shows our experimental results when applied to a finite-dimensional approximation of  100 players. We compare the trained neural network $\bar{v}^{*}_{N,\eta}$ with the explicit one-dimensional solution, as well as Monte Carlo estimates along two parameterised curves in the domain. We report a residual loss of  $\mathcal{L}_{res} = 2.45e-3$, averaged over 1000 random samples of quantized measures, as well as an HJB loss of $\mathcal{L}_{\text {HJB}} = 3.82e-4$. The HJB loss demonstrates a good fit of the neural network towards the finite-dimensional value function, and the residual loss indicates a good approximation of the mean field value function $v$ on the domain $\Pcal_2(\R)$. In terms of computational performance, the training time of the neural network $\bar{v}^{*}_{N,\eta}$ was approximately 7 hours.
After training, an average of 0.66 seconds is required for the evaluation of $v(t,\mu)$ for a given measure $\mu$. In comparison, generating an Monte Carlo estimate of the value function `online' directly with the control $\alpha^{*}_{\theta}$ requires approximately 2.67 seconds per evaluation. Therefore, for applications requiring knowledge of the value function across the whole domain $\Pcal_2(\R^d)$, for example the computation of level set $\{\mu \in \Pcal_2(\R): v(t,\mu) = 0\}$, significant savings in computational time can be achieved.

\FloatBarrier 
\subsection*{Acknowledgments}
The authors acknowledge financial support under the National Recovery and Resilience Plan (NRRP):
Probabilistic Methods for Energy
Transition – CUP G53D23006840001 - Grant Assignment Decree No. 1379 adopted on 01/09/2023
by MUR. Jonathan Tam is also supported by the G-Research research grant.

\bibliographystyle{abbrvurl}
\bibliography{ref}

\end{document}